\documentclass[11pt]{article}

\usepackage{mathrsfs}
\usepackage{amsfonts}
\usepackage{titlesec}
\usepackage{enumerate}
\usepackage[toc,page,title,titletoc,header]{appendix}
\usepackage{amsmath,amsthm,amssymb,anysize}
\usepackage[numbers,sort&compress]{natbib}
\usepackage{color}
\usepackage{graphicx}

\theoremstyle{definition}
\newtheorem{thm}{Theorem}
\newtheorem{lem}{Lemma}[section]

\newtheorem{exa}{Example}
\marginsize{3cm}{3cm}{2.5cm}{2.2cm}
\numberwithin{equation}{section}
\numberwithin{equation}{subsection}

\newcommand{\F}{\mathbb{F}}
\newcommand{\N}{\mathbb{N}}
\newcommand{\Z}{\mathbb{Z}}

\def\dim{\mathrm{dim}}

\title{\textsf{On cohomology of filiform Lie superalgebras}}
\author{Yong Yang $^{a}$ and Wende Liu $^{b}$ \footnote{Correspondence: wendeliu@ustc.edu.cn}\setcounter{footnote}{-1}
\\
\\ \small\textit{$^{a}$School of Mathematics, Jilin University, Changchun 130012, China}
\\ \small\textit{$^{b}$School of Mathematical Sciences, Harbin Normal University, Harbin 150025, China}
  }
\date{ }

\begin{document}
\maketitle

\begin{quotation}
\small\noindent \textbf{Abstract}:
 Suppose the ground field $\F$ is an algebraically closed field of characteristic different from 2, 3. We determine the Betti numbers and make a decomposition of the associative superalgebra of the cohomology for the model filiform Lie superalgebra. We also describe the associative superalgebra structures of the (divided power) cohomology for some low-dimensional filiform Lie superalgebras.

\vspace{0.2cm} \noindent{\textbf{Keywords}}: filiform Lie superalgebra; Betti number; associative superalgebra

\vspace{0.2cm} \noindent{\textbf{Mathematics Subject Classification 2010}}: 17B30, 17B56

\end{quotation}
\setcounter{section}{-1}
\section{Introduction}
In 1970, in the study of the reducibility of the varieties of nilpotent Lie algebras, Vergne introduced the concept of filiform Lie algebras and showed  that every filiform Lie algebra can be obtained
by an infinitesimal deformation of the model filiform Lie algebra $L_{n}$ (see \cite{1}). Since then, the study of the filiform Lie algebras, especially the model filiform Lie algebra, has become an important subject. Many conclusions on cohomology of the model filiform Lie algebra with coefficients in the trivial module have been obtained. For example,
the Betti numbers for $L_{n}$ with coefficients in the trivial module over a field of characteristic zero have been calculated in \cite{2,3,4}. A result, in \cite{5}, states that the filiform Lie algebras $L_{n}$ and $\mathfrak{m}_{2}(n)$ have the same Betti numbers over a field of characteristic two, which is different from the case of characteristic zero. Moreover, the first three Betti numbers of $L_{n}$ and $\mathfrak{m}_{2}(n)$ over $\mathbb{Z}_{2}$ have been calculated in \cite{6}.
As what happens in the Lie case, every filiform Lie superalgebra can be obtained by an infinitesimal deformation of the model filiform Lie superalgebra $L_{n,m}$. Many conclusions on cohomology of the model filiform Lie superalgebra with coefficients in the adjoint module have been obtained. For example, Khakimdjanov and Navarro gave a complete description of the second cohomology of $L_{n,m}$ with coefficients in the adjoint module in \cite{7,8,9,10,11}. The first cohomology of $L_{n,m}$ with coefficients in the adjoint module has been described in \cite{12} by calculating the derivations. However, in the trivial module case, less of work is done for $L_{n,m}$.

Throughout this paper, the ground field $\F$ is an algebraically closed field of characteristic different from 2, 3 and all vector spaces, algebras are over $\F$. In the characteristic zero case, for any non-negative integer $k$, we make a decomposition of $\mathrm{H}^{k}(L_{n,m})$ by the Hochschild-Serre spectral sequences, moreover, we can describe completely the Betti number of $\mathrm{H}^{\bullet}(L_{n,m})$ and the superalgebra structures of the cohomology for some low-dimensional filiform Lie superalgebras. We also describe the graded superalgebra structures of the (divided power) cohomology for some low-dimensional filiform Lie superalgebras by the Hochschild-Serre spectral sequences to a certain ideal, in characteristic different from 2, 3.

\section{Filiform Lie superalgebras}

A \emph{Lie superalgebra} is a $\Z_2$-graded algebra whose multiplication satisfies
the skew-supersymmetry and the super Jacobi identity (see \cite{13}). For a Lie superalgebra $L$, we inductively define two sequences:
\begin{equation*}
L^0_{\bar 0}=L_{\bar 0}, \qquad L_{\bar 0}^{i+1}=[L_{\bar 0},L^i_{\bar 0}]
\end{equation*}
and
\begin{equation*}
L^0_{\bar 1}=L_{\bar 1}, \qquad L_{\bar 1}^{i+1}=[L_{\bar 0},L^i_{\bar 1}].
\end{equation*}
If there exists $(m,n)\in \N^2$ such that
$L_{\bar 0}^m=0$, $L_{\bar 0}^{m-1} \neq 0$ and
$L_{\bar 1}^n=0$, $L_{\bar 1}^{n-1} \neq 0$,
the
pair $(m,n)$ is called the \emph{super-nilindex} of $L$.
In particular, a Lie superalgebra $L$ is called \emph{filiform} if
its super-nilindex is $(\dim L_{\bar 0}-1, \dim L_{\bar 1})$ (see \cite{7}).

Denote by $\mathcal{F}_{n,m}$ the set of filiform Lie superalgebras of super-nilindex $(n,m)$.
Let $\mathcal{F}$ be a filiform Lie superalgebra over $\mathbb{C}$. If $\mathcal{F}\in \mathcal{F}_{n,m}$, there exists a basis $\{ X_0,X_1,\ldots ,X_n\mid Y_1,\ldots, Y_m\}$ of $\mathcal{F}$ such that:
\begin{eqnarray*}
&&[X_{0},X_{i}]=X_{i+1},\ 1\leq i\leq n-1,\quad [X_{0},X_{n}]=0;\\
&&[X_{1},X_{2}]\in \mathbb{C}X_{4}+\cdots+\mathbb{C}X_{n}; \\
&&[X_{0},Y_{i}]=Y_{i+1},\ 1\leq i\leq m-1,\quad [X_{0},Y_{m}]=0.
\end{eqnarray*}
Recall the following classifications up to isomorphism (see \cite{14}):

The classification of $\mathcal{F}_{1,2}$:

$(1)\ \mathcal{F}_{1,2}^{1}: [X_{0},Y_{1}]=Y_{2};$

$(2)\ \mathcal{F}_{1,2}^{2}: [X_{0},Y_{1}]=[X_{1},Y_{1}]=Y_{2};$

$(3)\ \mathcal{F}_{1,2}^{3}: [X_{0},Y_{1}]=Y_{2},\ [Y_{1},Y_{1}]=X_{1}.$

The classification of $\mathcal{F}_{2,2}$:

$(1)\ \mathcal{F}_{2,2}^{1}: [X_{0},X_{1}]=X_{2},\ [X_{0},Y_{1}]=Y_{2};$

$(2)\ \mathcal{F}_{2,2}^{2}: [X_{0},X_{1}]=2[Y_{1},Y_{2}]=X_{2},\ [X_{0},Y_{1}]=Y_{2},\ [Y_{1},Y_{1}]=X_{1};$

$(3)\ \mathcal{F}_{2,2}^{3}: [X_{0},X_{1}]=[Y_{1},Y_{1}]=X_{2},\ [X_{0},Y_{1}]=Y_{2};$

$(4)\ \mathcal{F}_{2,2}^{4}: [X_{0},X_{1}]=X_{2},\ [X_{0},Y_{1}]=[X_{1},Y_{1}]=Y_{2};$

$(5)\ \mathcal{F}_{2,2}^{5}: [X_{0},X_{1}]=[Y_{1},Y_{1}]=X_{2},\ [X_{0},Y_{1}]=[X_{1},Y_{1}]=Y_{2}.$


Let $L_{n,m}$ be the filiform Lie superalgebra with a homogeneous basis
$$\{ X_0,X_1,\ldots ,X_n\mid Y_1,\ldots, Y_m\}$$
and   Lie super-brackets are given by
\begin{equation*}
[X_0,X_i]=X_{i+1}, \; 1\leq i\leq n-1;  \quad
[X_0,Y_j]=Y_{j+1}, \; 1\leq j\leq m-1.
\end{equation*}
 We call
$L_{n,m}$ the {\emph{model filiform Lie superalgebra}} and
$\{ X_0,X_1,\dots ,X_n\mid Y_1,\dots, Y_m\}$ the {\emph{standard basis}} of $L_{n,m}$.

Obviously, we have:
$$\mathcal{F}_{1,2}^{1}=L_{1,2},\quad \mathcal{F}_{2,2}^{1}=L_{2,2}.$$

\section{(Divided power) cohomology}

In this section, we introduce the definitions of cohomology and divided power cohomology of $\mathfrak{g}$ with coefficients in the trivial module.  For more details, the reader is referred to \cite{13,15,16}.

Let $\mathfrak{g}$ be a finite-dimensional Lie superalgebra, denote by $\mathfrak{g}^{\ast}$ the dual superspace of $\mathfrak{g}$. Fix an ordered basis of $\mathfrak{g}$
\begin{equation}\label{1}
\{x_{1},\ldots,x_{m}\mid x_{m+1},\ldots,x_{m+n}\},
\end{equation}
where $|x_{1}|=\cdots=|x_{m}|=0$, $|x_{m+1}|=\cdots=|x_{m+n}|=1$, and write
\begin{equation*}
\{x_{1}^{\ast},\ldots,x_{m}^{\ast}\mid x_{m+1}^{\ast},\ldots,x_{m+n}^{\ast}\},
\end{equation*}
for the dual basis.

For $k\in \mathbb{Z}$, we let $\bigwedge^{k}\mathfrak{g}^{\ast}$ be the $k$-th super-exterior product of $\mathfrak{g}^{\ast}$. Let $\bigwedge^{\bullet}\mathfrak{g}^{\ast}=\bigoplus\limits_{k\in \mathbb{N}_{0}}\bigwedge^{k}\mathfrak{g}^{\ast}$. Then $\bigwedge^{\bullet}\mathfrak{g}^{\ast}$ can be viewed as a $\mathfrak{g}$-module in a natural manner. Note that $\bigwedge^{\bullet}\mathfrak{g}^{\ast}$ also has a $\mathbb{Z}$-grading structure given by setting $\|x_{1}\|=\cdots=\|x_{m+n}\|=1$. Hereafter $\|x\|$ denotes the $\mathbb{Z}$-degree of a $\mathbb{Z}$-homogeneous element $x$ in a $\mathbb{Z}$-graded superspace. Let $d:\bigwedge^{\bullet}\mathfrak{g}^{\ast}\longrightarrow \bigwedge^{\bullet}\mathfrak{g}^{\ast}$ be the linear operator induced by the dual of the Lie superalgebra bracket map $\mathfrak{g}^{\ast}\longrightarrow \bigwedge^{2}\mathfrak{g}^{\ast}$. Then
\begin{eqnarray*}
&&d(1)=0, \\
&&d(x_{i}^{\ast})=\sum\limits_{1\leq k<l\leq m+n}(-1)^{|x_{k}^{\ast}||x_{l}^{\ast}|}a_{kl}^{i}x_{k}^{\ast}x_{l}^{\ast}-\frac{1}{2}\sum\limits_{m+1\leq k\leq m+n}a_{kk}^{i}x_{k}^{\ast^{2}},\  1\leq i\leq m+n, \\
&&d(x\wedge y)=d(x)\wedge y+(-1)^{\|x\|}x\wedge d(y),\ x,y\in \bigwedge^{\bullet}\mathfrak{g}^{\ast},
\end{eqnarray*}
where $a_{kl}^{i}$, $1\leq i,k,l\leq m+n$, are the structure constants of $\mathfrak{g}$ with respect to the basis (\ref{1}). Then $d$ is a $\mathfrak{g}$-module homomorphism and
$$d^{2}=0,\quad |d|=0,\quad \|d\|=1.$$
Denote by $\mathrm{H}^{\bullet}(\mathfrak{g})$ the $cohomology$ of $\mathfrak{g}$ defined by the cochain complex $(\bigwedge^{\bullet}\mathfrak{g}^{\ast},d)$. Note that $\bigwedge^{\bullet}\mathfrak{g}^{\ast}$ is a $\mathbb{Z}$-graded associative superalgebra in a natural manner, which induces a $\mathbb{Z}$-graded associative superalgebra structure on $\mathrm{H}^{\bullet}(\mathfrak{g})$. In particular, the dimension of $\mathrm{H}^{\mathrm{k}}(\mathfrak{g})$ is called the $k$-th $\emph{Betti}$ $number$.

Over a field of prime characteristic $p>2$, we introduce the definition of divided power cohomology of $\mathfrak{g}$. For a multi-index $\underline{r}=(r_{1},\ldots,r_{m+n})$, where $r_{1},\ldots,r_{m}$ are 0 or 1, and $r_{m+1},\ldots,r_{m+n}$ are non-negative integers, we set
$$u_{i}^{(r_{i})}=\frac{x_{i}^{r_{i}}}{r_{i}!}\quad \mathrm{and} \quad u^{(\underline{r})}=\prod_{i=1}^{m+n}u_{i}^{(r_{i})}.$$
Clearly, their multiplication relations are
\begin{equation}\label{2}
u^{(\underline{r})}u^{(\underline{s})}=\bigg(\prod\limits_{i=1}^{m}\mathrm{min}(1,2-r_{i}-s_{i})\bigg)(-1)^{\sum\limits_{j=1}^{m}\sum\limits_{k=m+1}^{m+n}r_{k}s_{j}+\sum\limits_{1\leq j<k\leq m}r_{k}s_{j}}
\left(
                    \begin{array}{c}
                        \underline{r}+\underline{s} \\
                          \underline{r} \\
                       \end{array}
                     \right)
u^{(\underline{r}+\underline{s})},
\end{equation}
where
\begin{equation*}
\left(
                    \begin{array}{c}
                        \underline{r}+\underline{s} \\
                          \underline{r} \\
                       \end{array}
                     \right)
=\prod_{i=m+1}^{m+n}\left(
                    \begin{array}{c}
                        r_{i}+s_{i} \\
                          r_{i} \\
                       \end{array}
                     \right).
\end{equation*}
Fix two n-tuples of positive integers $\underline{t}=(t_{1},\ldots,t_{n})$ and $\pi=(\pi_{1},\ldots,\pi_{n})$, where $\pi_{i} = p^{t_{i}}$, $1\leq i\leq n$. Denote
\begin{equation*}
 \mathcal{O}(m,n;\underline{t})=\mathcal{O}(\mathfrak{g}^{\ast};\underline{t})=\mathrm{span}_{\F}\left(u^{(\underline{r})}\mid r_{i}\left\{
       \begin{array}{ll}
         <\pi_{i}, & \hbox{$m+1\leq i\leq m+n$;} \\
         =0\ \mathrm{or}\ 1, & \hbox{$1\leq i\leq m$.}
       \end{array}
     \right.
\right).
\end{equation*}
From Eq. (\ref{2}), $\mathcal{O}(\mathfrak{g}^{\ast};\underline{t})$ is a finite dimensional submodule and a graded subalgebra of $\bigwedge^{\bullet}\mathfrak{g}^{\ast}$. In particular,
$$\mathcal{O}(\mathfrak{g}^{\ast};\underline{t})=\bigoplus_{i=0}^{m-n+\sum\limits_{j=1}^{n}p^{t_{j}}}\mathcal{O}_{i}(\mathfrak{g}^{\ast};\underline{t}),$$
where $\mathcal{O}_{i}(\mathfrak{g}^{\ast};\underline{t})=\mathrm{span}_{\F}\left\{u^{(\underline{r})}\mid \sum\limits_{k=1}^{m+n}r_{k}=i\right\}$.
Let $d:\mathcal{O}(\mathfrak{g}^{\ast};\underline{t})\longrightarrow \mathcal{O}(\mathfrak{g}^{\ast};\underline{t})$ be the linear operator induced by the dual of the Lie superalgebra bracket map. Then
\begin{eqnarray*}
&&d(1)=0, \\
&&d(x_{i}^{\ast})=\sum\limits_{1\leq k<l\leq m+n}(-1)^{|x_{k}^{\ast}||x_{l}^{\ast}|}a_{kl}^{i}x_{k}^{\ast}x_{l}^{\ast}-\sum\limits_{m+1\leq k\leq m+n}a_{kk}^{i}x_{k}^{\ast^{(2)}},\  1\leq i\leq m+n, \\
&&d(u^{(\underline{r})})=\sum_{j=1}^{m+n}(-1)^{r_{1}+\ldots+r_{j-1}+|x_{j}^{\ast}|(r_{m+1}+\ldots+r_{j-1})}d(x_{j}^{\ast})u^{(\underline{r}-\varepsilon_{j})},
\end{eqnarray*}
where $\varepsilon_{j}=(\delta_{j,1},\ldots,\delta_{j,m+n})$, $a_{kl}^{i}$, $1\leq i,j,k,l\leq m+n$, are the structure constants of $\mathfrak{g}$ with respect to the basis (\ref{1}). Then $d$ is a $\mathfrak{g}$-module homomorphism and
$$d^{2}=0,\quad |d|=0,\quad \|d\|=1.$$
Denote by $\mathrm{DPH}^{\bullet}(\mathfrak{g})$ the $divided\ power\ cohomology$ of $\mathfrak{g}$ defined by the cochain complex $(\mathcal{O}(\mathfrak{g}^{\ast};\underline{t}),d)$. Note that $\mathcal{O}(\mathfrak{g}^{\ast};\underline{t})$ is a $\mathbb{Z}$-graded associative superalgebra in a natural manner, which induces a $\mathbb{Z}$-graded associative superalgebra structure on $\mathrm{DPH}^{\bullet}(\mathfrak{g})$.

Finally, we close this section with some remarks on the Hochschild-Serre spectral sequence (see \cite{13}). Suppose $I$ is an ideal of $\mathfrak{g}$. Then there is a convergent spectral sequence called the Hochschild-Serre spectral sequence such that
\begin{equation}\label{1.1}
\mathrm{E}_{2}^{k,s}=\mathrm{H}^{k}(\mathfrak{g}/I,\mathrm{H}^{s}(I))\Longrightarrow \mathrm{H}^{k+s}(\mathfrak{g}).
\end{equation}

\section{Characteristic zero}  \label{mainresult}

Throughout this section the ground field $\F$ is an algebraically closed field of characteristic zero.

\subsection{Model filiform Lie superalgebras}

Let $\mathcal{I}$ be a subspace of $L_{n,m}$ spanned by
$$\{ X_1,\ldots ,X_n\mid Y_1,\ldots, Y_m\}.$$
Obviously, $\mathcal{I}$ is an abelian ideal of $L_{n,m}$. Moreover, $\mathrm{H}^{\bullet}(L_{n,m})$ is described by using the Hochschild-Serre spectral sequence relative to $\mathcal{I}$.

\begin{lem} \label{2.1.1}
Let
\begin{align*}
D: \bigwedge^{\bullet}\mathcal{I}^{\ast} &\longrightarrow \bigwedge^{\bullet}\mathcal{I}^{\ast},\\
x  &\longmapsto -X_{0}\cdot x .
\end{align*}
Let $D_{k}=D\mid_{\bigwedge^{k}\mathcal{I}^{\ast}}$. Then
$\mathrm{H}^{k}(L_{n,m})=\mathrm{Ker}\ D_{k}\bigoplus\left(\mathbb{F}X_{0}^{\ast}\bigwedge\left(\bigwedge^{k-1}\mathcal{I}^{\ast}/\mathrm{Im}\ D_{k-1}\right)\right).$
\end{lem}
\begin{proof}
From Eq. (\ref{1.1}), we have
\begin{equation*}
\mathrm{E}_{\infty}^{p,k-p}\cong\left\{
                       \begin{array}{ll}
                        \mathrm{H}^{0}\left(\F X_{0},\bigwedge^{k}\mathcal{I}^{\ast}\right), & \hbox{$p=0;$} \\
                        \mathrm{H}^{1}\left(\F X_{0},\bigwedge^{k-1}\mathcal{I}^{\ast}\right), & \hbox{$p=1;$} \\
                        0, & \hbox{$p\neq0,1.$}
                        \end{array}
                       \right.
\end{equation*}
Then we have
\begin{equation*}
\mathrm{H}^{k}(L_{n,m})=\bigoplus_{p+q=k}\mathrm{E}_{\infty}^{p,q}=\mathrm{H}^{0}\left(\mathbb{F}X_{0},\bigwedge^{k}\mathcal{I}^{\ast}\right)\bigoplus \mathrm{H}^{1}\left(\mathbb{F}X_{0},\bigwedge^{k-1}\mathcal{I}^{\ast}\right).
\end{equation*}
By the definitions of the low cohomology (see \cite{13}), we have
\begin{eqnarray*}
&&\mathrm{H}^{0}\left(\mathbb{F}X_{0},\bigwedge^{k}\mathcal{I}^{\ast}\right)=\mathrm{Ker}\ D_{k},\\
&&\mathrm{Der}\left(\mathbb{F}X_{0},\bigwedge^{k-1}\mathcal{I}^{\ast}\right)=\mathbb{F}X_{0}^{\ast}\bigwedge \ \bigwedge^{k-1}\mathcal{I}^{\ast},\\
&&\mathrm{Inder}\left(\mathbb{F}X_{0},\bigwedge^{k-1}\mathcal{I}^{\ast}\right)=\mathbb{F}X_{0}^{\ast}\bigwedge\mathrm{Im}\ D_{k-1}.
\end{eqnarray*}
Moreover, we obtain that
$$\mathrm{H}^{k}(L_{n,m})=\mathrm{Ker}\ D_{k}\bigoplus\left(\mathbb{F}X_{0}^{\ast}\bigwedge\left(\bigwedge^{k-1}\mathcal{I}^{\ast}/\mathrm{Im}\ D_{k-1}\right)\right).$$
\end{proof}
\begin{thm} \label{2.1.2}
As a $\mathbb{Z}$-graded superalgebra, we have
$$\mathrm{H}^{\bullet}(L_{n,m})=\mathrm{Ker}\ D \ltimes \left(\mathbb{F}X_{0}^{\ast}\bigwedge\left(\bigwedge^{\bullet}\mathcal{I}^{\ast}/\mathrm{Im}\ D\right)\right).$$
\end{thm}
\begin{proof}
For any $x,y\in \mathrm{Ker}\ D$, $\overline{z}\in \bigwedge^{\bullet}\mathcal{I}^{\ast}/\mathrm{Im}\ D$, we have
\begin{align*}
D(x\wedge y)=-X_{0}\cdot (x\wedge y)&=-[(X_{0}\cdot x)\wedge y+x\wedge (X_{0}\cdot y)],\\
&=D(x)\wedge y+x\wedge D(y),\\
&=0,
\end{align*}
and
\begin{align*}
(X_{0}^{\ast}\wedge \overline{z})\wedge x&=X_{0}^{\ast}\wedge(\overline{z}\wedge x),\\
&=X_{0}^{\ast}\wedge(\overline{z\wedge x})
\in \mathbb{F}X_{0}^{\ast}\bigwedge\left(\bigwedge^{\bullet}\mathcal{I}^{\ast}/\mathrm{Im}\ D\right).
\end{align*}
Thus, $\mathrm{Ker}\ D$ is a subalgebra of $\mathrm{H}^{\bullet}(L_{n,m})$ and $\mathbb{F}X_{0}^{\ast}\bigwedge\left(\bigwedge^{\bullet}\mathcal{I}^{\ast}/\mathrm{Im}\ D\right)$ is an ideal of
$\mathrm{H}^{\bullet}(L_{n,m})$ with trivial multiplication. From Lemma \ref{2.1.1}, the proof is complete.
\end{proof}

 In order to calculate the dimension of $\mathrm{H}^{\bullet}(L_{n,m})$, it suffices to calculate the dimension of $\mathrm{Ker}\ D$.

 Let
$$f_{i}=\frac{1}{(n-i)!}X_{i}^{\ast},\ 1\leq i\leq n ,\quad g_{j}=\frac{1}{(m-j)!}Y_{j}^{\ast},\ 1\leq j\leq m .$$

Let
$x, y, h\in \mathrm{End}_{\F}(\mathcal{I}^{\ast}_{0})$, such that
\begin{eqnarray*}
&&x=D\mid_{\mathcal{I}^{\ast}_{0}}, \\
&&y(f_{i})=if_{i+1},\ 1\leq i\leq n-1, \quad y(f_{n})=0, \\
&&h(f_{i})=(n+1-2i)f_{i},\ 1\leq i\leq n.
\end{eqnarray*}

Let
$x', y', h'\in \mathrm{End}_{\F}(\mathcal{I}^{\ast}_{1})$, such that
\begin{eqnarray*}
&&x'=D\mid_{\mathcal{I}^{\ast}_{1}}, \\
&&y'(g_{j})=jg_{j+1},\ 1\leq j\leq m-1,\quad y'(g_{m})=0, \\
&&h'(g_{j})=(m+1-2j)g_{j},\ 1\leq j\leq m.
\end{eqnarray*}

Obviously, $\mathrm{span}_{\F}\{x,y,h\}$, $\mathrm{span}_{\F}\{x',y',h'\}$ are subalgebras of $\mathfrak{gl}(\mathcal{I}^{\ast}_{0})$ and $\mathfrak{gl}(\mathcal{I}^{\ast}_{1})$, respectively. Moreover, the following Lie algebra isomorphisms hold:
$$\mathrm{span}_{\F}\{x,y,h\}\cong\mathrm{span}_{\F}\{x',y',h'\}\cong \mathfrak{sl}(2).$$

By Weyl's Theorem and representation theory of $\mathfrak{sl}(2)$ (see \cite{17}), $\mathcal{I}^{\ast}_{0}$ and $\mathcal{I}^{\ast}_{1}$ are simple modules of $\mathfrak{sl}(2)$ and, for $k\geq0$, $\bigwedge^{k}\mathcal{I}^{\ast}$ is a completely reducible module of $\mathfrak{sl}(2)$. Moreover, we can obtain the following theorem.

\begin{lem} \label{2.1.3}
Suppose  $k\geq0$. Then
$$\mathrm{dim}\ \mathrm{Ker}\ D_{k}=\sum\limits_{k_{0}=0}^{k} \mathfrak{f}_{n,m}(k_{0},k-k_{0}),$$
where $\mathfrak{f}_{n,m}(k_{0},k-k_{0})=\mathrm{card} \bigg \{(i_{1},\ldots,i_{k_{0}},j_{1},\ldots,j_{k-k_{0}})\in\mathbb{Z}^{k}\mid 1\leq
i_{1}<\ldots<i_{k_{0}}\leq n, 1\leq j_{1}<\ldots<j_{k-k_{0}}\leq m+k-k_{0}-1,\sum\limits_{a=1}^{k_{0}}i_{a}+\sum\limits_{b=1}^{k-k_{0}}j_{b}=\bigg\lfloor\dfrac{k_{0}(n+1)+(k-k_{0})(m+1)}{2} \bigg\rfloor+\dfrac{(k-k_{0})(k-k_{0}-1)}{2}\bigg \} .$
\end{lem}
\begin{proof}
For $k\geq0$, set
\begin{align*}
\bigwedge^{k}\mathcal{I}^{\ast}=\bigoplus\limits_{i=1}^{r}V_{i},
\end{align*}
where $V_{1},\ldots,V_{r}$ are simple $\mathfrak{sl}(2)$-modules. Moreover,
$$\mathrm{Ker}\ D_{k}=\bigoplus_{i=1}^{r}\mathrm{Ker}\ D_{k}\bigcap V_{i}.$$

So we obtain that for any $v\in V_{i}$, $D_{k}(v)=0$ if and only if $v$ is a maximal vector of $V_{i}$. Moreover, the following conclusions hold:
\begin{equation}\label{1.3}
\mathrm{dim}\ \mathrm{Ker}\ D_{k}=r=\mathrm{dim}\ \left(\bigwedge^{k}\mathcal{I}^{\ast}\right)_{0}+\mathrm{dim}\ \left(\bigwedge^{k}\mathcal{I}^{\ast}\right)_{1},
\end{equation}
where $\left(\bigwedge^{k}\mathcal{I}^{\ast}\right)_{0}$, $\left(\bigwedge^{k}\mathcal{I}^{\ast}\right)_{1}$ are  weight spaces of weight $0$, $1$, respectively.

Let $\mathfrak{sl}(2)=\mathrm{span}_{\F}\{X_{+},H,X_{-}\}$, where $\F H$ is a Cartan subalgebra of $\mathfrak{sl}(2)$, and $f_{i_{1}}\wedge\ldots\wedge f_{i_{k_{0}}}\wedge g_{j_{1}}\wedge\ldots\wedge g_{j_{k-k_{0}}}$ is a standard basis of $\bigwedge^{k}\mathcal{I}^{\ast}$, where $0\leq k_{0}\leq k$, $i_{1}<\ldots<i_{k_{0}}\leq n$, $1\leq j_{1}\leq\ldots\leq j_{k-k_{0}}\leq m$. Note that

$$
H(f_{i_{1}}\wedge\ldots\wedge f_{i_{k_{0}}}\wedge g_{j_{1}}\wedge\ldots\wedge g_{j_{k-k_{0}}})=\left(k_{0}(n+1)+(k-k_{0})(m+1)-2\left( \sum_{a=1}^{k_{0}}i_{a}+\sum_{b=1}^{k-k_{0}}j_{b}\right)\right)$$
$$\quad \quad\quad\quad \quad\quad\quad\quad f_{i_{1}}\wedge\ldots\wedge f_{i_{k_{0}}}\wedge g_{j_{1}}\wedge\ldots\wedge g_{j_{k-k_{0}}}.
$$

From Eq. ($\ref{1.3}$), it is sufficient to calculate the dimensions of $\left(\bigwedge^{k}\mathcal{I}^{\ast}\right)_{0}$ and $\left(\bigwedge^{k}\mathcal{I}^{\ast}\right)_{1}$. We consider the following cases:

$\mathbf{Case\ 1:}$ $k_{0}(n+1)+(k-k_{0})(m+1)$ is even. Then

$$\left(k_{0}(   n+1)+(k-k_{0})(m+1)-2\left( \sum_{a=1}^{k_{0}}i_{a}+\sum_{b=1}^{k-k_{0}}j_{b}\right)\right)=0,$$ if and only if $\sum\limits_{a=1}^{k_{0}}i_{a}+\sum\limits_{b=1}^{k-k_{0}}j_{b}=\dfrac{k_{0}(n+1)+(k-k_{0})(m+1)}{2}$.

$\mathbf{Case\ 2:}$ $k_{0}(n+1)+(k-k_{0})(m+1)$ is odd. Then

$$\left(k_{0}(   n+1)+(k-k_{0})(m+1)-2\left( \sum_{a=1}^{k_{0}}i_{a}+\sum_{b=1}^{k-k_{0}}j_{b}\right)\right)=1,$$ if and only if $\sum\limits_{a=1}^{k_{0}}i_{a}+\sum\limits_{b=1}^{k-k_{0}}j_{b}=\dfrac{k_{0}(n+1)+(k-k_{0})(m+1)-1}{2}$.

Thus, the conclusion holds.
\end{proof}

\begin{thm} \label{2.1.4}
Suppose  $k\geq0$. Then
$$\mathrm{dim}\ \mathrm{H}^{k}(L_{n,m})=\sum\limits_{k_{0}=0}^{k}\mathfrak{f}_{n,m}(k_{0},k-k_{0})+\sum\limits_{k_{0}=0}^{k-1}\mathfrak{f}_{n,m}(k_{0},k-k_{0}-1).$$
\end{thm}
\begin{proof}
It follows from Lemmas \ref{2.1.1} and \ref{2.1.3}.
\end{proof}

In order to characterize  the superalgebra structure of $\mathrm{H}^{\bullet}(L_{n,m})$, we make a $\mathbb{Z}$-gradation of $\bigwedge^{k}\mathcal{I}^{\ast}$ for any $k\geq0$.

For any $0\leq s\leq k$, let $l$ such that $\alpha_{k,s}\leq l\leq \beta_{k,s}$, where $\alpha_{k,s}=\frac{s(s+1)}{2}+k-s$, $\beta_{k,s}=\frac{s(2n-s+1)}{2}+(k-s)m$.

Let $\left(\bigwedge^{k}\mathcal{I}^{\ast}\right)^{l}_{s}$ be the space spanned by
$$X_{i_{1}}^{\ast}\wedge\ldots \wedge X_{i_{s}}^{\ast}\wedge Y^{\ast^{\alpha_{1}}}_{1}\wedge\ldots \wedge Y^{\ast^{\alpha_{m}}}_{m},$$
where $1\leq i_{1}<\ldots< i_{s}\leq n,$ and $\alpha_{1},\ldots,\alpha_{m}\geq 0,$
satisfying that
$$\sum\limits_{j=1}^{m}\alpha_{j}=k-s,\quad \sum\limits_{a=1}^{s}i_{a}+\sum\limits_{b=1}^{m}\alpha_{b}b=l.$$

Let
\begin{align*}
D_{s}^{k,l}: \left(\bigwedge^{k}\mathcal{I}^{\ast}\right)^{l}_{s}&\longrightarrow \left(\bigwedge^{k}\mathcal{I}^{\ast}\right)^{l-1}_{s},\\
x&\longmapsto D_{k}(x).
\end{align*}

Obviously,
$$\bigwedge^{k}\mathcal{I}^{\ast}=\bigoplus_{s=0}^{k}\bigoplus_{l=\alpha_{k,s}}^{\beta_{k,s}}\left(\bigwedge^{k}\mathcal{I}^{\ast}\right)^{l}_{s}.$$

Since the linear mapping $D_{k}$ is compatible with this graduation, we have:
\begin{eqnarray}
&&\mathrm{Ker}\ D_{k}=\bigoplus_{s=0}^{k}\bigoplus_{l=\alpha_{k,s}}^{\beta_{k,s}}\mathrm{Ker}\ D^{k,l}_{s},\\
&&\mathrm{Im}\  D_{k}=\bigoplus_{s=0}^{k}\bigoplus_{l=\alpha_{k,s}}^{\beta_{k,s}}\mathrm{Im}\ D^{k,l}_{s},
\end{eqnarray}
where $\alpha_{k,s}=\frac{s(s+1)}{2}+k-s$, $\beta_{k,s}=\frac{s(2n-s+1)}{2}+(k-s)m$.

We describe the superalgebra structures of the cohomology for some low-dimensional filiform Lie superalgebras by the decomposition in Theorem \ref{2.1.2}.

\begin{exa}
The following $\mathbb{Z}$-graded superalgebra isomorphism holds:
$$\mathrm{H}^{\bullet}(L_{1,2})\cong \mathcal{U}_{1,2}\ltimes \mathcal{V}_{1,2},$$
where $\mathcal{U}_{1,2}$ is an infinite-dimensional $\mathbb{Z}$-graded superalgebra with a $\mathbb{Z}_{2}$-homogeneous basis
$$\{\alpha_{1,i},\alpha_{i}\mid i\geq 0\},$$
satisfying that $|\alpha_{1,i}|=|\alpha_{i}|=\overline{i}$ (mod 2), $\|\alpha_{1,i}\|=i+1$, $\|\alpha_{i}\|=i$, and the multiplication is given by
$$\alpha_{i}\alpha_{j}=\alpha_{i+j},\quad \alpha_{1,i}\alpha_{j}=\alpha_{1,i+j},\ i,j\geq 0.$$
$\mathcal{V}_{1,2}$ is an infinite-dimensional $\mathbb{Z}$-graded superalgebra with a $\mathbb{Z}_{2}$-homogeneous basis
$$\{\alpha_{0,i},\alpha_{0,1,i}\mid i\geq 0\},$$
satisfying that $|\alpha_{0,i}|=|\alpha_{0,1,i}|=\overline{i}$ (mod 2), $\|\alpha_{0,i}\|=i+1$, $\|\alpha_{0,1,i}\|=i+2$, and the trivial multiplication. The multiplication between $\mathcal{U}_{1,2}$ and $\mathcal{V}_{1,2}$ is graded-supercommutative, and the multiplication is given by
$$\alpha_{0,i}\alpha_{1,0}=(-1)^{i}\alpha_{0,1,i},\quad\alpha_{0,i}\alpha_{0}=\alpha_{0,i}, \quad \alpha_{0,1,i}\alpha_{0}=\alpha_{0,1,i},\ i\geq0.$$
\begin{proof}
Note that
\begin{equation*}
\left(\bigwedge^{k}\mathcal{I}^{\ast}\right)^{l}_{s}=\left\{
                \begin{array}{ll}
                  \F Y^{\ast^{2k-l}}_{1}\wedge Y^{\ast^{l-k}}_{2}, & \hbox{$s=0$;} \\
               \F X_{1}^{\ast}\wedge Y^{\ast^{2k-l-1}}_{1}\wedge Y^{\ast^{l-k}}_{2}, & \hbox{$s=1;$} \\
                   0, & \hbox{else.} \\
                \end{array}
              \right.
\end{equation*}
Moreover, we have
\begin{equation*}
\mathrm{Ker}\ D^{k,l}_{s}=\left\{
                \begin{array}{ll}
                  \F Y^{\ast^{k}}_{1}, & \hbox{$s=0,l=k$;} \\
               \F X_{1}^{\ast}\wedge Y^{\ast^{k-1}}_{1}, & \hbox{$s=1,l=k;$} \\
                   0, & \hbox{else.} \\
                \end{array}
              \right.
\end{equation*}
\begin{equation*}
\mathrm{Im}\ D^{k,l}_{s}=\left\{
                \begin{array}{ll}
                  \F Y^{\ast^{2k-l+1}}_{1}\wedge Y^{\ast^{l-k-1}}_{2}, & \hbox{$s=0,k+1\leq l\leq2k$;} \\
               \F X_{1}^{\ast}\wedge Y^{\ast^{2k-l}}_{1}\wedge Y^{\ast^{l-k-1}}_{2}, & \hbox{$s=1,k+1\leq l\leq 2k-1;$} \\
                   0, & \hbox{else.} \\
                \end{array}
              \right.
\end{equation*}
From Theorem \ref{2.1.2} and Eqs. (3.1.2), (3.1.3), $\mathrm{H}^{\bullet}(L_{1,2})$ has a basis:
$$Y_{1}^{\ast^{i}},\quad X_{1}^{\ast}\wedge Y_{1}^{\ast^{i}},\quad X_{0}^{\ast}\wedge Y_{2}^{\ast^{i}},\quad X_{0}^{\ast}\wedge X_{1}^{\ast}\wedge Y_{2}^{\ast^{i}}, \ i\geq 0.$$
The conclusion can be obtained by a direct calculation.
\end{proof}
\end{exa}

\begin{exa}
The following $\mathbb{Z}$-graded superalgebra isomorphism holds:
$$\mathrm{H}^{\bullet}(L_{2,1})\cong \mathcal{U}_{2,1}\ltimes \mathcal{V}_{2,1},$$
where $\mathcal{U}_{2,1}$ is an infinite-dimensional $\mathbb{Z}$-graded superalgebra with a $\mathbb{Z}_{2}$-homogeneous basis
$$\{\alpha_{i},\alpha_{1,i},\alpha_{1,2,i}\mid i\geq 0\},$$
satisfying that $|\alpha_{i}|=|\alpha_{1,i}|=|\alpha_{1,2,i}|=\overline{i}$ (mod 2), $\|\alpha_{i}\|=i$, $\|\alpha_{1,i}\|=i+1$, $\|\alpha_{1,2,i}\|=i+2$, and the multiplication is given by
$$\alpha_{i}\alpha_{j}=\alpha_{i+j},\quad \alpha_{1,i}\alpha_{j}=\alpha_{1,i+j},\quad \alpha_{1,2,i}\alpha_{j}=\alpha_{1,2,i+j},\ i,j\geq 0.$$
$\mathcal{V}_{2,1}$ is an infinite-dimensional $\mathbb{Z}$-graded superalgebra with a $\mathbb{Z}_{2}$-homogeneous basis
$$\{\alpha_{0,i},\alpha_{0,2,i},\alpha_{0,1,2,i}\mid i\geq 0\},$$
satisfying that $|\alpha_{0,i}|=|\alpha_{0,2,i}|=|\alpha_{0,1,2,i}|=\overline{i}$ (mod 2), $\|\alpha_{0,i}\|=i+1$, $\|\alpha_{0,2,i}\|=i+2$, $\|\alpha_{0,1,2,i}\|=i+3$, and the trivial multiplication. The multiplication between $\mathcal{U}_{2,1}$ and $\mathcal{V}_{2,1}$ is graded-supercommutative, and the multiplication is given by
$$\alpha_{0,i}\alpha_{j}=\alpha_{0,i+j},\quad\alpha_{0,2,i}\alpha_{j}=\alpha_{0,2,i+j}, \quad \alpha_{0,1,2,i}\alpha_{j}=\alpha_{0,1,2,i+j},$$
$$\alpha_{0,2,i}\alpha_{1,j}=(-1)^{i+1}\alpha_{0,1,2,i+j},\quad\alpha_{0,i}\alpha_{1,2,j}=\alpha_{0,1,2,i+j},\ i,j\geq0.$$
\begin{proof}
Note that
\begin{equation*}
\left(\bigwedge^{k}\mathcal{I}^{\ast}\right)^{l}_{s}=\left\{
                \begin{array}{ll}
                  \F Y^{\ast^{l}}_{1}, & \hbox{$s=0$;} \\
               \F X_{1}^{\ast}\wedge Y^{\ast^{l-1}}_{1}, & \hbox{$s=1,l=k;$} \\
               \F X_{2}^{\ast}\wedge Y^{\ast^{l-2}}_{1}, & \hbox{$s=1,l=k+1;$} \\
               \F X_{1}^{\ast}\wedge X_{2}^{\ast}\wedge Y^{\ast^{l-3}}_{1}, & \hbox{$s=2,l=k+1;$} \\
                   0, & \hbox{else.} \\
                \end{array}
              \right.
\end{equation*}
Moreover, we have
\begin{equation*}
\mathrm{Ker}\ D^{k,l}_{s}=\left\{
                \begin{array}{ll}
                  \F Y^{\ast^{l}}_{1}, & \hbox{$s=0$;}\\
               \F X_{1}^{\ast}\wedge Y^{\ast^{l-1}}_{1}, & \hbox{$s=1,l=k;$} \\
               \F X_{1}^{\ast}\wedge X_{2}^{\ast}\wedge Y^{\ast^{l-3}}_{1}, & \hbox{$s=2,l=k+1;$} \\
                   0, & \hbox{else.} \\
                \end{array}
              \right.
\end{equation*}
\begin{equation*}
\mathrm{Im}\ D^{k,l}_{s}=\left\{
                \begin{array}{ll}
                  \F X^{\ast}_{1}\wedge Y^{\ast^{l-2}}_{1}, & \hbox{$s=1,l=k+1$;} \\
                 0, & \hbox{else.} \\
                \end{array}
              \right.
\end{equation*}
From Theorem \ref{2.1.2} and Eqs. (3.1.2), (3.1.3), $\mathrm{H}^{\bullet}(L_{2,1})$ has a basis:
$$Y_{1}^{\ast^{i}},\quad X_{1}^{\ast}\wedge Y_{1}^{\ast^{i}},\quad X_{1}^{\ast}\wedge X_{2}^{\ast}\wedge Y^{\ast^{i}}_{1},\quad X_{0}^{\ast}\wedge Y_{1}^{\ast^{i}},\quad X_{0}^{\ast}\wedge X_{2}^{\ast}\wedge Y_{1}^{\ast^{i}},\quad X_{0}^{\ast}\wedge X_{1}^{\ast} \wedge X_{2}^{\ast}\wedge Y_{1}^{\ast^{i}},\ i\geq 0.$$
The conclusion can be obtained by a direct calculation.
\end{proof}
\end{exa}

\begin{exa}
The following $\mathbb{Z}$-graded superalgebra isomorphism holds:
$$\mathrm{H}^{\bullet}(L_{1,3})\cong \mathcal{U}_{1,3}\ltimes \mathcal{V}_{1,3},$$
where $\mathcal{U}_{1,3}$ is an infinite-dimensional $\mathbb{Z}$-graded superalgebra with a $\mathbb{Z}_{2}$-homogeneous basis
$$\{\alpha_{i,j},\alpha_{1,i,j},\mid i,j\geq 0\},$$
satisfying that $|\alpha_{i,j}|=|\alpha_{1,i,j}|=\overline{2j+i}$ (mod 2), $\|\alpha_{i,j}\|=2j+i$, $\|\alpha_{1,i,j}\|=2j+i+1$, and the multiplication is given by
$$\alpha_{i,j}\alpha_{i',j'}=\alpha_{i+i',j+j'},\quad \alpha_{1,i,j}\alpha_{i',j'}=\alpha_{1,i+i',j+j'},\ i,i',j,j'\geq 0.$$
$\mathcal{V}_{1,3}$ is an infinite-dimensional $\mathbb{Z}$-graded superalgebra with a $\mathbb{Z}_{2}$-homogeneous basis
$$\{\alpha_{0,i,j},\alpha_{0,1,i,j}\mid i,j\geq 0\},$$
satisfying that $|\alpha_{0,i,j}|=|\alpha_{0,1,i,j}|=\overline{2i+j}$ (mod 2), $\|\alpha_{0,i,j}\|=2i+j+1$, $\|\alpha_{0,1,i,j}\|=2i+j+2$, and the trivial multiplication. The multiplication between $\mathcal{U}_{1,3}$ and $\mathcal{V}_{1,3}$ is graded-supercommutative, and the multiplication is given by
$$\alpha_{0,i,j}\alpha_{i',j'}=\rho_{i,j,i',j'}\alpha_{0,i+i'+j',j-i'},$$
$$\alpha_{0,i,j}\alpha_{1,i',j'}=(-1)^{2i+j}
\alpha_{0,1,i,j}\alpha_{i',j'}=(-1)^{2i+j}\rho_{i,j,i',j'}\alpha_{0,i+i'+j',j-i'},\ i,j\geq0,$$
where $\rho_{i,j,i',j'}=\prod\limits_{s=0}^{i'-1}\frac{s-j}{2(i+j'+s)+1}+\sum\limits_{t=i'+1}^{i'+j'}\left(\prod\limits_{a=1}^{t-i'}\frac{2(a-1-j')}{a}\prod\limits_{b=0}^{t-1}\frac{b-j-t+i'}{2(i+i'+j'-t+b)+1}\right)$.
\begin{proof}
Note that
\begin{equation*}
\left(\bigwedge^{k}\mathcal{I}^{\ast}\right)^{l}_{s}=\left\{
                \begin{array}{ll}
                  \bigoplus\limits_{i=2k-l}^{\big\lfloor\frac{3k-l}{2}\big\rfloor}\F Y^{i}_{1}\wedge Y^{3k-l-2i}_{2}\wedge Y^{i+l-2k}_{3}, & \hbox{$s=0, k\leq l\leq 2k$;} \\
               \bigoplus\limits_{i=0}^{\big\lfloor\frac{3k-l}{2}\big\rfloor}\F Y^{i}_{1}\wedge Y^{3k-l-2i}_{2}\wedge Y^{i+l-2k}_{3}, & \hbox{$s=0, 2k+1 \leq l\leq 3k$} \\
               \bigoplus\limits_{i=2k-l-1}^{\big\lfloor\frac{3k-l}{2}\big\rfloor-1}\F X_{1}\wedge Y^{i}_{1}\wedge Y^{3k-l-2i-2}_{2}\wedge Y^{i+l-2k+1}_{3}, & \hbox{$s=1,k\leq l\leq 2k-1;$} \\
               \bigoplus\limits_{i=0}^{\big\lfloor\frac{3k-l}{2}\big\rfloor-1}\F X_{1}\wedge Y^{i}_{1}\wedge Y^{3k-l-2i-2}_{2}\wedge Y^{i+l-2k+1}_{3}, & \hbox{$s=1,2k\leq l\leq 3k-2;$} \\
                   0, & \hbox{else.} \\
                \end{array}
              \right.
\end{equation*}
From a direct computation, we have the following conclusion:

If $2k+1\leq l\leq 3k$, or $k\leq l\leq 2k$ and $3k-l$ is odd, we have $\mathrm{Ker}\ D^{k,l}_{0}=0$.

If $k\leq l\leq 2k$, and $3k-l$ is even, we have $\mathrm{dim}\ \mathrm{Ker}\ D^{k,l}_{0}=1$ and  $\mathrm{Ker}\ D^{k,l}_{0}$ has a basis:
$$Y^{\ast^{2k-l}}_{1}\wedge Y^{\ast^{l-k}}_{2}+\sum\limits_{i=2k-l+1}^{\frac{3k-l}{2}}\lambda_{i}^{k,l} Y^{\ast^{i}}_{1}\wedge Y^{\ast^{3k-l-2i}}_{2}\wedge Y^{\ast^{i+l-2k}}_{3},$$
where $\lambda_{i}^{k,l}=\prod\limits_{j=1}^{i-2k+l}\frac{2(j-1)+k-l}{j}$. Moreover, we have

\begin{equation*}
\left(\bigwedge^{k}\mathcal{I}^{\ast}\right)^{l-1}_{0}/\mathrm{Im}\ D^{k,l}_{0}=\left\{
                \begin{array}{ll}
                  \F Y^{\ast^{3k-l+1}}_{2}\wedge Y^{\ast^{l-2k-1}}_{3}, & \hbox{$2k+1\leq l\leq3k$ and $3k-l$ is odd;} \\
                 0, & \hbox{else.} \\
                \end{array}
              \right.
\end{equation*}
From Theorem \ref{2.1.2} and Eqs. (3.1.2), (3.1.3), $\mathrm{H}^{\bullet}(L_{1,3})$ has a basis:
$$Y_{1}^{\ast^{i}}\wedge Y_{2}^{\ast^{2j}}+\sum_{s=i+1}^{i+j}\lambda_{s}^{2j+i,4j+i}Y_{1}^{\ast^{s}}\wedge Y_{2}^{\ast^{2(i+j-s)}}\wedge Y_{3}^{\ast^{s-i}},$$
$$X_{1}^{\ast}\wedge Y_{1}^{\ast^{i}}\wedge Y_{2}^{\ast^{2j}}+\sum_{s=i+1}^{i+j}\lambda_{s}^{2j+i,4j+i}X_{1}^{\ast}\wedge Y_{1}^{\ast^{s}}\wedge Y_{2}^{\ast^{2(i+j-s)}}\wedge Y_{3}^{\ast^{s-i}},$$
$$X_{0}^{\ast}\wedge Y_{2}^{\ast^{2i}}\wedge Y_{3}^{\ast^{j}},\ X_{0}^{\ast}\wedge X_{1}^{\ast}\wedge Y_{2}^{\ast^{2i}}\wedge Y_{3}^{\ast^{j}},\ i,j\geq 0.$$
The conclusion can be obtained by a direct calculation.
\end{proof}
\end{exa}
\begin{exa}
The following $\mathbb{Z}$-graded superalgebra isomorphism holds:
$$\mathrm{H}^{\bullet}(L_{3,1})\cong \mathcal{U}_{3,1}\ltimes \mathcal{V}_{3,1},$$
where $\mathcal{U}_{3,1}$ is an infinite-dimensional $\mathbb{Z}$-graded superalgebra with a $\mathbb{Z}_{2}$-homogeneous basis
$$\{\alpha_{i},\alpha_{1,i},\alpha_{1,2,i},\alpha_{1,2,3,i}\mid i\geq 0\},$$
satisfying that $|\alpha_{i}|=|\alpha_{1,i}|=|\alpha_{1,2,i}|=|\alpha_{1,2,3,i}|=\overline{i}$ (mod 2), $\|\alpha_{i}\|=i$, $\|\alpha_{1,i}\|=i+1$, $\|\alpha_{1,2,i}\|=i+2$, $\|\alpha_{1,2,3,i}\|=i+3$, and the multiplication is given by
$$\alpha_{i}\alpha_{j}=\alpha_{i+j},\quad \alpha_{1,i}\alpha_{j}=\alpha_{1,i+j},\quad \alpha_{1,2,i}\alpha_{j}=\alpha_{1,2,i+j},\quad \alpha_{1,2,3,i}\alpha_{j}=\alpha_{1,2,3,i+j},\ i,j\geq 0.$$
$\mathcal{V}_{3,1}$ is an infinite-dimensional $\mathbb{Z}$-graded superalgebra with a $\mathbb{Z}_{2}$-homogeneous basis
$$\{\alpha_{0,i},\alpha_{0,3,i},\alpha_{0,2,3,i},\alpha_{0,1,2,3,i}\mid i\geq 0\},$$
satisfying that $|\alpha_{0,i}|=|\alpha_{0,3,i}|=|\alpha_{0,2,3,i}|=|\alpha_{0,1,2,3,i}|=\overline{i}$ (mod 2), $\|\alpha_{0,i}\|=i+1$, $\|\alpha_{0,3,i}\|=i+2$, $\|\alpha_{0,2,3,i}\|=i+3$, $\|\alpha_{0,1,2,3,i}\|=i+4$, and the trivial multiplication. The multiplication between $\mathcal{U}_{3,1}$ and $\mathcal{V}_{3,1}$ is graded-supercommutative, and the multiplication is given by
$$\alpha_{0,i}\alpha_{j}=\alpha_{0,i+j},\quad\alpha_{0,3,i}\alpha_{j}=\alpha_{0,3,i+j}, \quad \alpha_{0,2,3,i}\alpha_{j}=\alpha_{0,2,3,i+j},\quad \alpha_{0,1,2,3,i}\alpha_{j}=\alpha_{0,1,2,3,i+j},$$
$$ \alpha_{0,2,3,i}\alpha_{1,j}= \alpha_{0,i}\alpha_{1,2,3,j}=(-1)^{i}\alpha_{0,3,i}\alpha_{1,2,j}=(-1)^{i}\alpha_{0,1,2,3,i+j},\ i,j\geq0.$$
\begin{proof}
Note that
\begin{equation*}
\left(\bigwedge^{k}\mathcal{I}^{\ast}\right)^{l}_{s}=\left\{
                \begin{array}{ll}
                  \F Y^{\ast^{l}}_{1}, & \hbox{$s=0$;} \\
               \F X_{1}^{\ast}\wedge Y^{\ast^{l-1}}_{1}, & \hbox{$s=1,l=k;$} \\
               \F X_{2}^{\ast}\wedge Y^{\ast^{l-2}}_{1}, & \hbox{$s=1,l=k+1;$} \\
                \F X_{3}^{\ast}\wedge Y^{\ast^{l-3}}_{1}, & \hbox{$s=1,l=k+2;$} \\
            \F X_{1}^{\ast}\wedge X_{2}^{\ast}\wedge Y^{\ast^{l-3}}_{1}, & \hbox{$s=2,l=k+1;$} \\
               \F X_{1}^{\ast}\wedge X_{3}^{\ast}\wedge Y^{\ast^{l-4}}_{1}, & \hbox{$s=2,l=k+2;$} \\
                \F X_{2}^{\ast}\wedge X_{3}^{\ast}\wedge Y^{\ast^{l-5}}_{1}, & \hbox{$s=2,l=k+3;$} \\
               \F X_{1}^{\ast}\wedge X_{2}^{\ast}\wedge X_{3}^{\ast}\wedge Y^{\ast^{l-6}}_{1}, & \hbox{$s=3,l=k+3;$} \\
                   0, & \hbox{else.} \\
                \end{array}
              \right.
\end{equation*}
Moreover, we have
\begin{equation*}
\mathrm{Ker}\ D^{k,l}_{s}=\left\{
                \begin{array}{ll}
                  \F Y^{\ast^{l}}_{1}, & \hbox{$s=0$;} \\
               \F X_{1}^{\ast}\wedge Y^{\ast^{l-1}}_{1}, & \hbox{$s=1,l=k;$} \\
            \F X_{1}^{\ast}\wedge X_{2}^{\ast}\wedge Y^{\ast^{l-3}}_{1}, & \hbox{$s=2,l=k+1;$} \\
               \F X_{1}^{\ast}\wedge X_{2}^{\ast}\wedge X_{3}^{\ast}\wedge Y^{\ast^{l-6}}_{1}, & \hbox{$s=3,l=k+3;$} \\
                   0, & \hbox{else.} \\
                \end{array}
              \right..
\end{equation*}
\begin{equation*}
\mathrm{Im}\ D^{k,l}_{s}=\left\{
                \begin{array}{ll}
               \F X_{1}^{\ast}\wedge Y^{\ast^{l-2}}_{1}, & \hbox{$s=1,l=k+1;$} \\
                \F X_{2}^{\ast}\wedge Y^{\ast^{l-3}}_{1}, & \hbox{$s=1,l=k+2;$} \\
               \F X_{1}^{\ast}\wedge X_{2}^{\ast}\wedge Y^{\ast^{l-4}}_{1}, & \hbox{$s=2,l=k+2;$} \\
                \F X_{1}^{\ast}\wedge X_{3}^{\ast}\wedge Y^{\ast^{l-5}}_{1}, & \hbox{$s=2,l=k+3;$} \\              0, & \hbox{else.} \\
                \end{array}
              \right..
\end{equation*}
From Theorem \ref{2.1.2} and Eqs. (3.1.2), (3.1.3), $\mathrm{H}^{\bullet}(L_{3,1})$ has a basis:
$$Y_{1}^{\ast^{i}},\quad X_{1}^{\ast}\wedge Y_{1}^{\ast^{i}},\quad X_{1}^{\ast}\wedge X_{2}^{\ast}\wedge Y_{1}^{\ast^{i}},\quad X_{1}^{\ast}\wedge X_{2}^{\ast}\wedge X_{3}^{\ast}\wedge Y_{1}^{\ast^{i}},$$
$$ X_{0}^{\ast}\wedge Y_{1}^{\ast^{i}},\quad X_{0}^{\ast}\wedge X_{3}^{\ast}\wedge Y_{1}^{\ast^{i}}, \quad X_{0}^{\ast}\wedge X_{2}^{\ast}\wedge X_{3}^{\ast}\wedge Y_{1}^{\ast^{i}},\quad X_{0}^{\ast}\wedge X_{1}^{\ast}\wedge X_{2}^{\ast}\wedge X_{3}^{\ast}\wedge Y_{1}^{\ast^{i}},\ i\geq 0.$$
The conclusion can be obtained by a direct calculation.
\end{proof}
\end{exa}

\begin{exa}
The following $\mathbb{Z}$-graded superalgebra isomorphism holds:
$$\mathrm{H}^{\bullet}(L_{2,2})\cong \mathcal{U}_{2,2}\ltimes \mathcal{V}_{2,2},$$
where $\mathcal{U}_{2,2}$ is an infinite-dimensional $\mathbb{Z}$-graded superalgebra with a $\mathbb{Z}_{2}$-homogeneous basis
$$\{\alpha_{i},\alpha_{1,i},\alpha_{1,2,i},\beta_{i}\mid i\geq 0\},$$
satisfying that $|\alpha_{i}|=|\alpha_{1,i}|=|\alpha_{1,2,i}|=\overline{i}$ (mod 2), $|\beta_{i}|=\overline{i+1}$ (mod 2), $\|\alpha_{i}\|=i$, $\|\alpha_{1,i}\|=i+1$, $\|\alpha_{1,2,i}\|=\|\beta_{i}\|=i+2$, and the multiplication is given by
$$\alpha_{i}\alpha_{j}=\alpha_{i+j},\quad \alpha_{1,i}\alpha_{j}=\alpha_{1,i+j},\quad \alpha_{1,2,i}\alpha_{j}=\alpha_{1,2,i+j},$$
$$\beta_{i}\alpha_{j}=\beta_{i+j},\quad \alpha_{1,i}\beta_{j}=(-1)^{i+1}\alpha_{1,2,i+j+1},\ i,j\geq 0.$$
$\mathcal{V}_{2,2}$ is an infinite-dimensional $\mathbb{Z}$-graded superalgebra with a $\mathbb{Z}_{2}$-homogeneous basis
$$\{\alpha_{0,i},\alpha_{0,1,i},\alpha_{0,2,i},\alpha_{0,1,2,i}\mid i\geq 0\},$$
satisfying that $|\alpha_{0,i}|=|\alpha_{0,2,i}|=|\alpha_{0,1,2,i}|=\overline{i}$ (mod 2), $|\alpha_{0,1,i}|=\overline{i+1}$ (mod 2), $\|\alpha_{0,i}\|=i+1$, $\|\alpha_{0,2,i}\|=i+2$, $\|\alpha_{0,1,i}\|=\|\alpha_{0,1,2,i}\|=i+3$, and the trivial multiplication. The multiplication between $\mathcal{U}_{2,2}$ and $\mathcal{V}_{2,2}$ is graded-supercommutative, and the multiplication is given by
$$\alpha_{0,i}\alpha_{0}=\alpha_{0,i},\quad\alpha_{0,1,i}\alpha_{0}=\alpha_{0,1,i}, \quad \alpha_{0,2,i}\alpha_{0}=\alpha_{0,2,i},\quad \alpha_{0,1,2,i}\alpha_{0}=\alpha_{0,1,2,i},$$
$$ \alpha_{0,2,i}\alpha_{1,0}= (-1)^{i+1}\alpha_{0,i}\alpha_{1,2,0}=(-1)^{i+1}\alpha_{0,1,2,i},$$
$$ \alpha_{0,2,i}\beta_{0}= (-1)^{i+1}\alpha_{0,1,2,i+1},\quad \alpha_{0,i+1}\alpha_{1,0}= (-1)^{i+1}\alpha_{0,1,i},$$
$$ \alpha_{0,i}\beta_{0}= (-1)^{i}\frac{i+2}{i+1}\alpha_{0,1,i},\quad \alpha_{0,2,i}\alpha_{1}= -\frac{1}{i+1}\alpha_{0,1,i},\ i\geq0.$$
\begin{proof}
Note that
\begin{equation*}
\left(\bigwedge^{k}\mathcal{I}^{\ast}\right)^{l}_{s}=\left\{
                \begin{array}{ll}
                  \F Y^{\ast^{2k-l}}_{1}\wedge Y^{\ast^{l-k}}_{2}, & \hbox{$s=0$;} \\
               \F X_{1}^{\ast}\wedge Y^{\ast^{2k-l-1}}_{1}\wedge Y^{\ast^{l-k}}_{2}\bigoplus\F X_{2}^{\ast}\wedge Y^{\ast^{2k-l}}_{1}\wedge Y^{\ast^{l-k-1}}_{2}, & \hbox{$s=1;$} \\
               \F X_{1}^{\ast}\wedge X_{2}^{\ast}\wedge Y^{\ast^{2k-l-1}}_{1}\wedge Y^{\ast^{l-k-1}}_{2}, & \hbox{$s=2;$} \\
                   0, & \hbox{else.} \\
                \end{array}
              \right.
\end{equation*}
Moreover, we have
\begin{equation*}
\mathrm{Ker}\ D^{k,l}_{s}=\left\{
                \begin{array}{ll}
                  \F Y^{\ast^{l}}_{1}, & \hbox{$s=0,l=k$;} \\
                 \F X_{1}^{\ast}\wedge Y^{\ast^{l-1}}_{1}, & \hbox{$s=1,l=k;$} \\
               \F\left(X_{1}^{\ast}\wedge Y^{\ast^{l-3}}_{1}\wedge Y^{\ast}_{2}-X_{2}^{\ast}\wedge Y^{\ast^{l-2}}_{1}\right), & \hbox{$s=1,l=k+1;$} \\
            \F X_{1}^{\ast}\wedge X_{2}^{\ast}\wedge Y^{\ast^{l-3}}_{1}, & \hbox{$s=2,l=k+1;$} \\
                   0, & \hbox{else.} \\
                \end{array}
              \right..
\end{equation*}
\begin{equation*}
\mathrm{Im}\ D^{k,l}_{s}=\left\{
                \begin{array}{ll}
               \F Y^{\ast^{2k-l+1}}_{1} \wedge Y^{\ast^{l-k-1}}_{2}, & \hbox{$s=0,k+1\leq l\leq 2k;$} \\
                \F X_{1}^{\ast}\wedge Y^{\ast^{l-2}}_{1}, & \hbox{$s=1,l=k+1;$} \\
                \F X_{1}^{\ast}\wedge Y^{\ast^{2k-l}}_{1}\wedge Y^{\ast^{l-k-1}}_{2}\bigoplus\F X_{2}^{\ast}\wedge Y^{\ast^{2k-l+1}}_{1}\wedge Y^{\ast^{l-k-2}}_{2}, & \hbox{$s=1,k+2\leq l\leq 2k;$} \\
               \F X_{1}^{\ast}\wedge X_{2}^{\ast}\wedge Y^{\ast^{2k-l}}_{1}\wedge Y^{\ast^{l-k-2}}_{2}, & \hbox{$s=2,k+2\leq l\leq 2k-1;$} \\                            0, & \hbox{else.} \\
                \end{array}
              \right..
\end{equation*}
From Theorem \ref{2.1.2} and Eqs. (3.1.2), (3.1.3), $\mathrm{H}^{\bullet}(L_{3,1})$ has a basis:
$$Y_{1}^{\ast^{i}},\quad X_{1}^{\ast}\wedge Y_{1}^{\ast^{i}},\quad X_{1}^{\ast}\wedge X_{2}^{\ast}\wedge Y_{1}^{\ast^{i}},\quad X_{1}^{\ast}\wedge  Y_{1}^{\ast^{i}}\wedge  Y_{2}-X_{2}^{\ast}\wedge  Y_{1}^{\ast^{i+1}},$$
$$ X_{0}^{\ast}\wedge Y_{2}^{\ast^{i}},\quad X_{0}^{\ast}\wedge X_{1}^{\ast}\wedge Y_{2}^{\ast^{i+1}}, \quad X_{0}^{\ast}\wedge X_{2}^{\ast}\wedge Y_{2}^{\ast^{i}},\quad X_{0}^{\ast}\wedge X_{1}^{\ast}\wedge X_{2}^{\ast}\wedge  Y_{2}^{\ast^{i}},\ i\geq 0.$$
The conclusion can be obtained by a direct calculation.
\end{proof}
\end{exa}

\subsection{Low-dimensional filiform Lie superalgebras}
In this section, we describe the cohomology of $\mathcal{F}_{1,2}$, $\mathcal{F}_{2,2}$ by using the Hochschild-Serre spectral sequence.

\begin{lem}\label{2.5}

For $\mathcal{F}_{s,2}^{t_{s}}$, $s=1,2$, $t_{1}=1,2,3$, $t_{2}=1,3,4,5$, let $\mathcal{I}_{s,2}^{t_{s}}=[\mathcal{F}_{s,2}^{t_{s}},\mathcal{F}_{s,2}^{t_{s}}]$. For $k\geq 0$, $0\leq i\leq k$, the following conclusions hold:

(1) $\mathrm{E}_{2}^{k-i,i}=\bigwedge^{k-i}\left(\mathcal{F}_{s,2}^{t_{s}}/\mathcal{I}_{s,2}^{t_{s}}\right)^{\ast}\bigotimes \bigwedge^{i}\left(\mathcal{I}_{s,2}^{t_{s}}\right)^{\ast}\Longrightarrow \mathrm{H}^{k}\left(\mathcal{F}_{s,2}^{t_{s}}\right).$

(2) $\mathrm{E}_{\infty}^{k-i,i}=\mathrm{E}_{3}^{k-i,i}.$
\end{lem}

\begin{proof}
(1) Since $\mathcal{I}_{s,2}^{t_{s}}\subseteq C\left(\mathcal{F}_{s,2}^{t_{s}}\right)$, the action of $\mathcal{F}_{s,2}^{t_{s}}/\mathcal{I}_{s,2}^{t_{s}}$ on $\left(\mathcal{I}_{s,2}^{t_{s}}\right)^{\ast}$ is trivial. Moreover, from Eq. (2.0.3), we have
$$\mathrm{E}_{2}^{k-i,i}=\mathrm{H}^{k-i}\left(\mathcal{F}_{s,2}^{t_{s}}/\mathcal{I}_{s,2}^{t_{s}},\mathrm{H}^{i}\left(\mathcal{I}_{s,2}^{t_{s}}\right)\right)=\bigwedge^{k-i}\left(\mathcal{F}_{s,2}^{t_{s}}/\mathcal{I}_{s,2}^{t_{s}}\right)^{\ast}\bigotimes \bigwedge^{i}\left(\mathcal{I}_{s,2}^{t_{s}}\right)^{\ast}\Longrightarrow \mathrm{H}^{k}\left(\mathcal{F}_{s,2}^{t_{s}}\right).$$
(2) From (1), we have
$$\mathrm{E}_{r}^{k-i,i}=\mathrm{E}_{3}^{k-i,i}, \quad r\geq 3.$$
Moreover, $\mathrm{E}_{\infty}^{k-i,i}=\mathrm{E}_{3}^{k-i,i}.$
\end{proof}

\begin{thm}
The following $\mathbb{Z}$-graded superalgebra isomorphisms hold:

(1) $\mathrm{H}^{\bullet}(\mathcal{F}_{1,2}^{2})\cong \mathcal{U}_{1,2}^{2}\ltimes \mathcal{V}_{1,2}^{2},$
where $\mathcal{U}_{1,2}^{2}$ is an infinite-dimensional $\mathbb{Z}$-graded superalgebra with a $\mathbb{Z}_{2}$-homogeneous basis
$$\{\alpha_{i},\alpha_{0,i}\mid i\geq 0\},$$
satisfying that
$$|\alpha_{i}|=|\alpha_{0,i}|=\overline{i}\ (\mathrm{mod}\ 2);\quad \|\alpha_{i}\|=i,\ \|\alpha_{0,i}\|=i+1,$$
and the multiplication is given by
$$\alpha_{i}\alpha_{j}=\alpha_{i+j},\quad \alpha_{0,i}\alpha_{j}=\alpha_{0,i+j},\ i,j\geq 0,$$
$\mathcal{V}_{1,2}^{2}$ is an infinite-dimensional $\mathbb{Z}$-graded superalgebra with a $\mathbb{Z}_{2}$-homogeneous basis
$$\{\alpha_{0,1}^{i},\alpha_{0,1,i}\mid i\geq 0\},$$
satisfying that
$$|\alpha_{0,1}^{i}|=|\alpha_{0,1,i}|=\overline{i}\ (\mathrm{mod}\ 2);\quad \|\alpha_{0,1}^{i}\|=i+1,\ \|\alpha_{0,1,i}\|=i+2,$$
and the trivial multiplication. The multiplication between $\mathcal{U}_{1,2}^{2}$ and $\mathcal{V}_{1,2}^{2}$ is graded-super
-commutative, and the multiplication is given by
$$\alpha_{0,1}^{i}\alpha_{0}=\alpha_{0,1}^{i},\quad\alpha_{0,1,i}\alpha_{0}=\alpha_{0,1,i}, \quad \alpha_{0,0}\alpha_{0,1}^{i}=\alpha_{0,1,i},\ i\geq0.$$
In particular,
\begin{equation*}
\mathrm{dim}\ \mathrm{H}^{k}(\mathcal{F}_{1,2}^{2})=\left\{
                \begin{array}{ll}
                  1, & \hbox{$k=0$;} \\
                3, & \hbox{$k=1;$} \\
                   4, & \hbox{$k\geq 2.$} \\
                \end{array}
              \right.
\end{equation*}

(2) $\mathrm{H}^{\bullet}(\mathcal{F}_{1,2}^{3})\cong \mathcal{U}_{1,2}^{3},$
where $\mathcal{U}_{1,2}^{3}$ is an infinite-dimensional $\mathbb{Z}$-graded superalgebra with a $\mathbb{Z}_{2}$-homogeneous basis
$$\{\alpha_{0},\alpha_{1},\alpha_{0,i},\beta_{i}\mid i\geq 0\},$$
satisfying that
$$|\alpha_{0}|=\overline{0}\ (\mathrm{mod}\ 2),\ |\alpha_{1}|=\overline{1}\ (\mathrm{mod}\ 2),\ |\alpha_{0,i}|=|\beta_{i}|=\overline{i}\ (\mathrm{mod}\ 2);$$
$$ \|\alpha_{0}\|=0,\ \|\alpha_{1}\|=1,\ \|\alpha_{0,i}\|=i+1,\ \|\beta_{i}\|=i+2,$$
and the multiplication is given by
$$\alpha_{0}\alpha_{0}=\alpha_{0},\quad \alpha_{0}\alpha_{1}=\alpha_{1},\quad \alpha_{0}\alpha_{0,i}=\alpha_{0,i},\quad \alpha_{0}\beta_{i}=\beta_{i},\ i\geq 0.$$
In particular,
\begin{equation*}
\mathrm{dim}\ \mathrm{H}^{k}(\mathcal{F}_{1,2}^{3})=\left\{
                \begin{array}{ll}
                  1, & \hbox{$k=0$;} \\
                   2, & \hbox{$k\geq 1.$} \\
                \end{array}
              \right.
\end{equation*}

(3) $\mathrm{H}^{\bullet}(\mathcal{F}_{2,2}^{3})\cong \mathcal{U}_{2,2}^{3},$
where $\mathcal{U}_{2,2}^{3}$ is an infinite-dimensional $\mathbb{Z}$-graded superalgebra with a $\mathbb{Z}_{2}$-homogeneous basis
$$\{\alpha_{0},\alpha_{1},\beta_{1},\alpha_{1,1},\alpha_{0,i},\beta_{0,i},\alpha_{0,1,i},\beta_{0,1,i}\mid i\geq 0\},$$
satisfying that
$$|\alpha_{0}|=|\alpha_{1}|=\overline{0}\ (\mathrm{mod}\ 2),\ |\beta_{1}|=|\alpha_{1,1}|=\overline{1}\ (\mathrm{mod}\ 2),\ |\alpha_{0,i}|=|\alpha_{0,1,i}|=|\beta_{0,i}|=|\beta_{0,1,i}|=\overline{i}\ (\mathrm{mod}\ 2);$$
$$\|\alpha_{0}\|=0,\ \|\alpha_{1}\|=\|\beta_{1}\|=1,\ \|\alpha_{1,1}\|=2,\ \|\alpha_{0,i}\|=i+1,\ \|\alpha_{0,1,i}\|=\|\beta_{0,i}\|=i+2,\ \|\beta_{0,1,i}\|=i+3,$$
and the multiplication is given by
$$\alpha_{0}\alpha_{0}=\alpha_{0},\quad \alpha_{0}\alpha_{1}=\alpha_{1},\quad \alpha_{0}\beta_{1}=\beta_{1}, \quad \alpha_{0}\alpha_{1,1}=\alpha_{1,1},$$
$$\alpha_{0}\alpha_{0,i}=\alpha_{0,i},\quad \alpha_{0}\alpha_{0,1,i}=\alpha_{0,1,i},\quad \alpha_{0}\beta_{0,i}=\beta_{0,i},\quad \alpha_{0}\beta_{0,1,i}=\beta_{0,1,i},$$
$$ \alpha_{1}\beta_{1}=\alpha_{1,1},\quad \alpha_{1}\alpha_{0,i}=-\alpha_{0,1,i},\quad \alpha_{1}\beta_{0,i}=\beta_{0,1,i},\quad \beta_{1}\beta_{1}=2\alpha_{0,1,0},$$
$$\beta_{1}\beta_{0,i}=2\alpha_{0,1,i+1},\quad \beta_{0,i}\beta_{0,j}=2\alpha_{0,1,i+j+2},\ i\geq 0.$$
In particular,
\begin{equation*}
\mathrm{dim}\ \mathrm{H}^{k}(\mathcal{F}_{2,2}^{3})=\left\{
                \begin{array}{ll}
                  1, & \hbox{$k=0$;} \\
                   3, & \hbox{$k=1;$} \\
                  4, & \hbox{$k\geq 2.$} \\
                \end{array}
              \right.
\end{equation*}

(4) $\mathrm{H}^{\bullet}(\mathcal{F}_{2,2}^{4})\cong \mathcal{U}_{2,2}^{4}\ltimes \mathcal{V}_{2,2}^{4},$
where $\mathcal{U}_{2,2}^{4}$ is an infinite-dimensional $\mathbb{Z}$-graded superalgebra with a $\mathbb{Z}_{2}$-homogeneous basis
$$\{\alpha_{i},\beta_{i}\mid i\geq 0\},$$
satisfying that
$$|\alpha_{i}|=\overline{i}\ (\mathrm{mod}\ 2),\ |\beta_{i}|=\overline{i+1}\ (\mathrm{mod}\ 2);\quad \|\alpha_{i}\|=i,\ \|\beta_{i}\|=i+2,$$
and the multiplication is given by
$$\beta_{i}\alpha_{0}=\beta_{i},\quad \alpha_{i}\alpha_{j}=\alpha_{i+j},\ i,j\geq 0.$$
$\mathcal{V}_{2,2}^{4}$ is an infinite-dimensional $\mathbb{Z}$-graded superalgebra with a $\mathbb{Z}_{2}$-homogeneous basis
$$\{\alpha,\beta,\alpha_{0,i},\alpha_{0,1,i},\alpha_{0,2,i},\alpha_{0,1,2,i},\beta_{1,i},\beta_{2,i}\mid i\geq 0\},$$
satisfying that
$$|\alpha|=|\beta|=0\ (\mathrm{mod}\ 2),\ |\alpha_{0,i}|=|\alpha_{0,2,i}|=|\alpha_{0,1,2,i}|=\overline{i}\ (\mathrm{mod}\ 2),$$
$$|\alpha_{0,1,i}|=|\beta_{1,i}|=|\beta_{2,i}|=\overline{i+1}\ (\mathrm{mod}\ 2);$$
$$\|\alpha\|=1,\ \|\beta\|=2,\ \|\alpha_{0,i}\|=i+1,\ \|\alpha_{0,2,i}\|=\|\beta_{1,i}\|=i+2,$$
$$ \|\alpha_{0,1,i}\|=\|\alpha_{0,1,2,i}\|=\|\beta_{2,i}\|=i+3,$$
and the multiplication is given by
$$\alpha_{0,0}\beta_{2,i}=\alpha_{0,1,2,i+1},\quad \alpha_{0,0}\beta=\alpha_{0,1,2,0},\quad\alpha_{0,0}\beta_{1,0}=\alpha_{0,2,1}-\alpha_{0,1,0},$$
$$\alpha_{0,i}\beta_{1,j+1}=(-1)^{i}\alpha_{0,2,i+j+2},\quad\alpha_{0,i+1}\beta_{1,j}=(-1)^{i+1}\alpha_{0,2,i+j+2},$$
$$\alpha_{0,2,0}\alpha=-\alpha_{0,1,2,0},\quad \alpha\beta_{1,0}=-\alpha_{0,1,0}-\alpha_{0,2,1},$$
$$\alpha_{0,2,0}\beta_{1,0}=\alpha_{0,1,2,1},\quad \beta_{2,i}\alpha=(-1)^{i}\alpha_{0,1,2,i+1},$$
$$\beta_{2,i}\beta_{1,0}=(-1)^{i+1}\alpha_{0,1,2,i+2},\quad \alpha\beta_{1,i+1}=-\alpha_{0,2,i+2}\ i,j\geq 0.$$
The multiplication between $\mathcal{U}_{2,2}^{4}$ and $\mathcal{V}_{2,2}^{4}$ is graded-supercommutative, and the multiplication is given by
$$\alpha_{0,1,i}\alpha_{0}=\alpha_{0,1,i},\quad \alpha_{0,1,2,i}\alpha_{0}=\alpha_{0,1,2,i},\quad \beta_{2,i}\alpha_{0}=\beta_{2,i},\quad \alpha\alpha_{0}=\alpha,\quad\beta\alpha_{0}=\beta,$$
$$\alpha_{0,0}\beta_{i}=\alpha_{0,1,i},\quad \alpha_{0,i}\alpha_{j}=\alpha_{0,i+j},\quad \alpha_{0,2,i}\alpha_{j}=\alpha_{0,2,i+j},\quad \beta_{i}\alpha=(-1)^{i+1}\alpha_{0,1,i},$$
$$ \beta_{2,i}\alpha_{1}=-\frac{1}{i+2}\alpha_{0,1,i+1},\quad \alpha\alpha_{i+1}=-\alpha_{0,i+1},\quad \alpha_{i}\beta_{1,j}=(-1)^{i}\beta_{1,i+j},$$
$$\alpha_{0,2,0}\beta_{i}=-\alpha_{0,1,2,i+1}, \quad\beta_{i}\beta=\alpha_{0,1,2,i+1},\quad \beta_{i}\beta_{1,0}=(-1)^{i}\frac{i+3}{i+2}\alpha_{0,1,i+1},$$
$$\beta\alpha_{1}=-\alpha_{0,1,0}-\alpha_{0,2,1},\quad\beta\alpha_{i+2}=-\alpha_{0,2,i+2},\ i,j\geq 0.$$
In particular,
\begin{equation*}
\mathrm{dim}\ \mathrm{H}^{k}(\mathcal{F}_{2,2}^{4})=\left\{
                \begin{array}{ll}
                  1, & \hbox{$k=0$;} \\
                   3, & \hbox{$k=1;$} \\
                  6, & \hbox{$k=2;$} \\
                  8, & \hbox{$k\geq 3.$} \\
                \end{array}
              \right.
\end{equation*}

(5) $\mathrm{H}^{\bullet}(\mathcal{F}_{2,2}^{5})\cong \mathcal{U}_{2,2}^{5},$
where $\mathcal{U}_{2,2}^{5}$ is an infinite-dimensional $\mathbb{Z}$-graded superalgebra with a $\mathbb{Z}_{2}$-homogeneous basis
$$\{\alpha_{0},\alpha_{1},\alpha_{0,1,i},\alpha_{0,1}^{i},\beta_{1,i},\beta_{2,i}\mid i\geq 0\},$$
satisfying that
$$|\alpha_{0}|=|\alpha_{1}|=\overline{0}\ (\mathrm{mod}\ 2),\ |\alpha_{0,1,i}|=|\alpha_{0,1}^{i}|=\overline{i}\ (\mathrm{mod}\ 2),\ |\beta_{1,i}|=|\beta_{2,i}|=\overline{i+1}\ (\mathrm{mod}\ 2);$$
$$\|\alpha_{0}\|=0,\ \|\alpha_{1}\|=1,\ \|\alpha_{0,1}^{i}\|=\|\beta_{1,i}\|=i+1,\ \|\alpha_{0,1,i}\|=\|\beta_{2,i}\|=i+2,$$
 and the multiplication is given by
$$\alpha_{0}\alpha_{0}=\alpha_{0},\quad \alpha_{0}\alpha_{1}=\alpha_{1},\quad \alpha_{0}\alpha_{0,1,i}=\alpha_{0,1,i},\quad \alpha_{0}\alpha_{0,1}^{i}=\alpha_{0,1}^{i},\quad \alpha_{0}\beta_{1,i}=\beta_{1,i},\quad \alpha_{0}\beta_{2,i}=\beta_{2,i},$$
$$\alpha_{1}\alpha_{0,1}^{i}=\alpha_{0,1,i},\quad \alpha_{1}\beta_{1,i}=\beta_{2,i},\quad \beta_{1,i_{1}}\beta_{1,i_{2}}=2\alpha_{0,1,i_{1}+i_{2}},\ i,i_{1},i_{2}\geq 0.$$
In particular,
\begin{equation*}
\mathrm{dim}\ \mathrm{H}^{k}(\mathcal{F}_{2,2}^{5})=\left\{
                \begin{array}{ll}
                  1, & \hbox{$k=0$;} \\
                   3, & \hbox{$k=1;$} \\
                  4, & \hbox{$k\geq 2.$} \\
                \end{array}
              \right.
\end{equation*}
\end{thm}

\begin{proof}
(1) For $k\geq 0$, $0\leq i\leq k$, consider the mapping
\begin{align*}
  d_{2}^{k-i,i}: \bigwedge^{k-i}(\mathcal{F}_{1,2}^{2}/\F Y_{2})^{\ast}\bigotimes \F Y_{2}^{\ast^{i}}&\longrightarrow \bigwedge^{k-i+2}(\mathcal{F}_{1,2}^{2}/\F Y_{2})^{\ast}\bigotimes \F Y_{2}^{\ast^{i-1}}, \\
 f\otimes Y_{2}^{\ast^{i}}& \longmapsto (-1)^{\|f\|}if\wedge d(Y_{2}^{\ast})\otimes Y_{2}^{\ast^{i-1}},
\end{align*}
where
\begin{equation*}
d(Y_{2}^{\ast})=(X_{0}^{\ast}+X_{1}^{\ast})\wedge Y_{1}^{\ast}.
\end{equation*}
By Lemma \ref{2.5}, we have
\begin{equation*}
\mathrm{E}^{k-i,i}_{\infty}=\left\{
                \begin{array}{ll}
                 \frac{\bigwedge^{k}(\mathcal{F}_{1,2}^{2}/\F Y_{2})^{\ast}}{\bigwedge^{k-2}(\mathcal{F}_{1,2}^{2}/\F Y_{2})^{\ast}\bigwedge \F d(Y_{2}^{\ast})}, & \hbox{$i=0$;} \\
                   \frac{\F(X_{0}^{\ast}+X_{1}^{\ast})\wedge Y_{1}^{\ast^{k-i-1}}\otimes Y_{2}^{\ast^{i}}\bigoplus\F X_{0}^{\ast}\wedge X_{1}^{\ast}\wedge Y_{1}^{\ast^{k-i-2}}\otimes Y_{2}^{\ast^{i}}}{\bigwedge^{k-2-i}(\mathcal{F}_{1,2}^{2}/\F Y_{2})^{\ast}\bigwedge \F d(Y_{2}^{\ast})\bigotimes \F Y_{2}^{\ast^{i}}}, & \hbox{$1\leq i\leq k-2$;} \\
                \F(X_{0}^{\ast}+X_{1}^{\ast})\bigotimes Y_{2}^{\ast^{k-1}}, & \hbox{$i=k-1$;} \\               0, & \hbox{$i=k$.}\\
                \end{array}
              \right.
\end{equation*}
From $\mathrm{H}^{k}(\mathcal{F}_{1,2}^{2})=\bigoplus\limits_{i=0}^{k}\mathrm{E}^{k-i,i}_{\infty}$, we can obtain the conclusion.

The proofs of $(2), (3), (4), (5)$ are similar to $(1)$.
\end{proof}

\begin{lem}\label{4}
For $\mathcal{F}_{2,2}^{2}$, let $I=\mathrm{span}_{\F}\{X_{1},X_{2},Y_{1},Y_{2}\}$.

$(1)$ For $k\geq 0$, the following conclusion holds:
$$\mathrm{H}^{k}(F_{2,2}^{2})\cong \mathrm{H}^{0}\left(\F X_{0},\mathrm{H}^{k}(I)\right)\bigoplus \mathrm{H}^{1}\left(\F X_{0}, \mathrm{H}^{k-1}(I)\right).$$

$(2)$ For $k\geq 2$, $\mathrm{H}^{k}(I)$ has a basis:

$$Y_{2}^{\ast^{k}}, \quad X_{1}^{\ast}\wedge Y_{2}^{\ast^{k-1}}-X_{2}^{\ast}\wedge Y_{1}^{\ast}\wedge Y_{2}^{\ast^{k-2}}.$$

$(3)$ For $k\geq 2$, $\mathrm{H}^{0}\left(\F X_{0},\mathrm{H}^{k}(I)\right)$ has a basis:
$$Y_{2}^{\ast^{k}}, \quad X_{1}^{\ast}\wedge Y_{2}^{\ast^{k-1}}-X_{2}^{\ast}\wedge Y_{1}^{\ast}\wedge Y_{2}^{\ast^{k-2}}.$$

$(4)$ For $k\geq 3$, $\mathrm{H}^{1}\left(\F X_{0}, \mathrm{H}^{k-1}(I)\right)$ has a basis:
$$X_{0}^{\ast}\wedge Y_{2}^{\ast^{k-1}}, \quad X_{0}^{\ast}\wedge X_{1}^{\ast}\wedge Y_{2}^{\ast^{k-2}}-X_{0}^{\ast}\wedge X_{2}^{\ast}\wedge Y_{1}^{\ast}\wedge Y_{2}^{\ast^{k-3}}.$$

\begin{proof}
(1) Note $I$ is an ideal of $\mathcal{F}_{2,2}^{2}$. From Eq. (\ref{1.1}), we have
\begin{equation*}
\mathrm{E}^{i,k-i}_{\infty}\cong\left\{
                \begin{array}{ll}
                 \mathrm{H}^{0}\left(\F X_{0},\mathrm{H}^{k}(I)\right), & \hbox{$i=0$;} \\
                \mathrm{H}^{1}\left(\F X_{0},\mathrm{ H}^{k-1}(I)\right), & \hbox{$i=1$;} \\
            0, & \hbox{else.}\\
                \end{array}
              \right.
\end{equation*}
Moreover, we have
$$\mathrm{H}^{k}(F_{2,2}^{2})=\bigoplus\limits_{i=0}^{k}\mathrm{E}^{i,k-i}_{\infty}\cong \mathrm{H}^{0}\left(\F X_{0},\mathrm{H}^{k}(I)\right)\bigoplus \mathrm{H}^{1}\left(\F X_{0}, \mathrm{H}^{k-1}(I)\right).$$

(2) We use the Hochschild-Serre spectral sequence relative to the ideal $I_{1}=[I,I]$. Note that $I_{1}\subseteq C(I)$, we have
\begin{equation*}
\mathrm{E}^{k-i,i}_{2}=\left\{
                \begin{array}{ll}
                 \bigwedge^{k-i}(I/I_{1})^{\ast}\bigotimes\bigwedge^{i}I_{1}^{\ast}, & \hbox{$0\leq i\leq 2$;} \\
            0, & \hbox{else.}\\
                \end{array}
              \right.
\end{equation*}
Moreover, we have
\begin{equation*}
\mathrm{E}^{k-i,i}_{\infty}=\mathrm{E}^{k-i,i}_{3}=\left\{
                \begin{array}{ll}
                 \F Y_{2}^{\ast^{k}}, & \hbox{$i=0$;} \\
                 \F \left(X_{1}^{\ast}\wedge Y_{2}^{\ast^{k-1}}-X_{2}^{\ast}\wedge Y_{1}^{\ast}\wedge Y_{2}^{\ast^{k-2}}\right), & \hbox{$i=1$;} \\
            0, & \hbox{else.}\\
                \end{array}
              \right.
\end{equation*}
Form $\mathrm{H}^{k}(I)=\bigoplus\limits_{i=0}^{k}\mathrm{E}^{k-i,i}_{\infty}$, we can obtain the conclusion.

(3) By the definitions of the low cohomology (see \cite{10}), we have
$$\mathrm{H}^{0}\left(\F X_{0},\mathrm{H}^{k}(I)\right)=\mathrm{H}^{k}(I).$$

(4) From (3), we have
$$\mathrm{H}^{1}\left(\F X_{0}, \mathrm{H}^{k-1}(I)\right)=\F X_{0}^{\ast}\bigwedge \mathrm{H}^{k-1}(I).$$
\end{proof}
\end{lem}
\begin{thm}
The following $\mathbb{Z}$-graded superalgebra isomorphism holds:
$$\mathrm{H}^{\bullet}(F_{2,2}^{2})\cong\mathcal{U}^{2}_{2,2},$$
where $\mathcal{U}_{2,2}^{2}$ is an infinite-dimensional $\mathbb{Z}$-graded superalgebra with a $\mathbb{Z}_{2}$-homogeneous basis
$$\{\alpha_{0},\alpha_{1},\alpha_{0,i},\beta_{0},\beta_{1,i},\beta_{2,i},\beta_{3,i}\mid i\geq 0\},$$
satisfying that
$$|\alpha_{0}|=\overline{0}\ (\mathrm{mod}\ 2),\ |\alpha_{1}|=|\beta_{0}|=\overline{1}\ (\mathrm{mod}\ 2),$$
$$ |\alpha_{0,i}|=|\beta_{1,i}|=|\beta_{3,i}|=\overline{i}\ (\mathrm{mod}\ 2),\ |\beta_{2,i}|=\overline{i+1}\ (\mathrm{mod}\ 2);$$
$$\|\alpha_{0}\|=0,\ \|\alpha_{1}\|=1,\ \|\beta_{0}\|=2,\ \|\alpha_{0,i}\|=i+1,\ \|\beta_{1,i}\|=i+2,\ \|\beta_{2,i}\|=\|\beta_{3,i}\|=i+3,$$
 and the multiplication is given by
$$\alpha_{0}\alpha_{0}=\alpha_{0},\quad \alpha_{0}\alpha_{1}=\alpha_{1},\quad \alpha_{0}\alpha_{0,i}=\alpha_{0,i},\quad \alpha_{0}\beta_{0}=\beta_{0},\quad \alpha_{0}\beta_{1,i}=\beta_{1,i},\quad \alpha_{0}\beta_{2,i}=\beta_{2,i},$$
$$\alpha_{0}\beta_{3,i}=\beta_{3,i},\quad \beta_{1,i}\alpha_{1}=2\beta_{2,i},\quad \alpha_{0,i}\beta_{0}=(-1)^{i}\beta_{2,i},\quad \alpha_{0,i}\beta_{1,j}=\alpha_{0,i+j+2},$$
$$\alpha_{0,i}\beta_{3,j}=(-1)^{i}\beta_{2,i+j+1},\quad \beta_{1,i}\beta_{0}=(-1)^{i}\beta_{3,i+1},\quad \beta_{1,i}\beta_{1,j}=\beta_{1,i+j+2},$$
$$\beta_{2,i}\beta_{1,j}=\beta_{2,i+j+2},\quad \beta_{3,i}\beta_{1,j}=\beta_{3,i+j+2},\ i,j\geq 0.$$
In particular,
\begin{equation*}
\mathrm{dim}\ \mathrm{H}^{k}(\mathcal{F}_{2,2}^{2})=\left\{
                \begin{array}{ll}
                  1, & \hbox{$k=0$;} \\
                   2, & \hbox{$k=1;$} \\
                 3, & \hbox{$k=2;$} \\
                  4, & \hbox{$k\geq 3.$} \\
                \end{array}
              \right.
\end{equation*}
\begin{proof}
For $k\geq 3$, by Lemma \ref{4}, let
\begin{equation*}
\begin{split}
\Phi: \mathrm{H}^{0}\left(\F X_0,\mathrm{H}^{k}(I)\right)&\longrightarrow \mathrm{H}^{k}(\mathcal{F}_{2,2}^{2}),\\
                             Y_{2}^{\ast^{k}} &\longmapsto Y_{2}^{\ast^{k}}-2kX_{0}^{\ast}\wedge X_{2}^{\ast}\wedge Y_{2}^{\ast^{k-2}},\\
         X_{1}^{\ast}\wedge Y_{2}^{\ast^{k-1}}-X_{2}^{\ast}\wedge Y_{1}^{\ast}\wedge Y_{2}^{\ast^{k-2}}&\longmapsto X_{1}^{\ast}\wedge Y_{2}^{\ast^{k-1}}-X_{2}^{\ast}\wedge Y_{1}^{\ast}\wedge Y_{2}^{\ast^{k-2}}\\
                                  &\quad\quad +2(k-2)X_{0}^{\ast}\wedge X_{1}^{\ast}\wedge X_{2}^{\ast}\wedge Y_{2}^{\ast^{k-3}},
\end{split}
\end{equation*}
and
\begin{equation*}
\begin{split}
\Psi: \mathrm{H}^{1}\left(\F X_0,\mathrm{H}^{k-1}(I)\right)&\longrightarrow \mathrm{H}^{k}(\mathcal{F}_{2,2}^{2}),\\
                 X_{0}^{\ast}\wedge Y_{2}^{\ast^{k-1}} &\longmapsto X_{0}^{\ast}\wedge Y_{2}^{\ast^{k-1}},\\
              X_{0}^{\ast}\wedge X_{1}^{\ast}\wedge Y_{2}^{\ast^{k-2}}-X_{0}^{\ast}\wedge X_{2}^{\ast}\wedge Y_{1}^{\ast}\wedge Y_{2}^{\ast^{k-3}}
&\longmapsto X_{0}^{\ast}\wedge X_{1}^{\ast}\wedge Y_{2}^{\ast^{k-2}}-X_{0}^{\ast}\wedge X_{2}^{\ast}\wedge Y_{1}^{\ast}\wedge Y_{2}^{\ast^{k-3}}.
\end{split}
\end{equation*}
Then $\Phi$ and $\Psi$ are injective  linear mappings, and $\mathrm{Im}\ \Phi \cap\mathrm{Im}\ \Psi=0$. Moreover, we have $$\mathrm{H}^{k}(\mathcal{F}_{2,2}^{2})=\Phi\left(\mathrm{H}^{0}\left(\F X_0,\mathrm{H}^{k}(I)\right)\right)\bigoplus \Psi\left(\mathrm{H}^{1}\left(\F X_0,\mathrm{H}^{k-1}(I)\right)\right). $$
Thus, $\mathrm{H}^{\bullet}(\mathcal{F}_{2,2}^{2})$ has a basis:
$$1,\quad Y_{1}^{\ast},\quad X_{0}^{\ast}\wedge Y_{2}^{\ast^{i}},\quad X_{1}^{\ast}\wedge Y_{2}^{\ast}-X_{2}^{\ast}\wedge Y_{1}^{\ast},\quad Y_{2}^{\ast^{i+2}}-2(i+2)X_{0}^{\ast}\wedge X_{2}^{\ast}\wedge Y_{2}^{\ast^{i}},$$
$$ X_{0}^{\ast}\wedge X_{1}^{\ast}\wedge Y_{2}^{\ast^{i+1}}-X_{0}^{\ast}\wedge X_{2}^{\ast}\wedge Y_{1}^{\ast}\wedge Y_{2}^{\ast^{i}},$$
$$\quad X_{1}^{\ast}\wedge Y_{2}^{\ast^{i+2}}-X_{2}^{\ast}\wedge Y_{1}^{\ast}\wedge Y_{2}^{\ast^{i+1}}+2(i+1)X_{0}^{\ast}\wedge X_{1}^{\ast}\wedge X_{2}^{\ast}\wedge Y_{2}^{\ast^{i}},\ i\geq 0.$$
The conclusion can be obtained by a direct calculation.
\end{proof}

\end{thm}

\section{Characteristic $p>3$}  \label{mainresult}

Throughout this section the ground field $\F$ is an algebraically closed field of characteristic $p>3$.

\begin{lem}\label{2.7}

For $\mathcal{F}_{s,2}^{t_{s}}$, $s=1,2$, $t_{1}=1,2,3$, $t_{2}=1,3,4,5$, let $\mathcal{I}_{s,2}^{t_{s}}=[\mathcal{F}_{s,2}^{t_{s}},\mathcal{F}_{s,2}^{t_{s}}]$. For $k\geq 0$, $0\leq i\leq k$, the following conclusions hold:

(1) $\mathrm{E}_{\mathrm{2}}^{k-i,i}=\mathcal{O}_{k-i}\left(\left(\mathcal{F}_{s,2}^{t_{s}}/\mathcal{I}_{s,2}^{t_{s}}\right)^{\ast};\underline{t}\right)\bigotimes \mathcal{O}_{i}\left(\left(\mathcal{I}_{s,2}^{t_{s}}\right)^{\ast};\underline{t}\right)\Longrightarrow \mathrm{DPH}^{k}(\mathcal{F}_{s,2}^{t_{s}}).$

(2) $\mathrm{E}_{\infty}^{k-i,i}=\mathrm{E}_{3}^{k-i,i}.$

\begin{proof}
The proof is similar to Lemma \ref{2.5}.
\end{proof}
\end{lem}

\subsection{Divided power cohomology of $\mathcal{F}_{1,2}$}

\begin{thm}\label{3}
The following $\mathbb{Z}$-graded superalgebra isomorphisms hold:

(1)  $\mathrm{DPH}^{\bullet}(\mathcal{F}_{1,2}^{1})\cong \mathcal{U}_{1,2}^{1}(\underline{t}),$
where $\mathcal{U}_{1,2}^{1}(\underline{t})$ is a $p^{t_{1}-1}(4p^{t_{2}}+2p-2)$-dimensional $\mathbb{Z}$-graded superalgebra with a $\mathbb{Z}_{2}$-homogeneous basis
$$\{\alpha_{0}^{i,j},\alpha_{1,k},\alpha_{k},\alpha_{0,1}^{i,j},\alpha_{1}^{s,t},\alpha^{s,t}\mid 0\leq i\leq p^{t_{1}-1}-1,0\leq j\leq p^{t_{2}}-1,0\leq k\leq p^{t_{1}}-1,$$
$$\quad\quad\quad 1\leq s\leq p^{t_{1}-1},1\leq t\leq p^{t_{2}}-1\},$$
satisfying that
$$|\alpha_{0}^{i,j}|=|\alpha_{0,1}^{i,j}|=\overline{i+j}\ (\mathrm{mod}\ 2),\ |\alpha_{k}|=|\alpha_{1,k}|=\overline{k}\ (\mathrm{mod}\ 2),\ |\alpha^{s,t}|=|\alpha_{1}^{s,t}|=\overline{s+t-1}\ (\mathrm{mod}\ 2);$$
$$\|\alpha_{0}^{i,j}\|=ip+j+1,\ \|\alpha_{0,1}^{i,j}\|=ip+j+2,\ \|\alpha_{k}\|=k,\ \|\alpha_{1,k}\|=k+1,\ \|\alpha^{s,t}\|=sp+t-1,\ \|\alpha_{1}^{s,t}\|=sp+t,$$
 and the multiplication is given by
$$\alpha_{0}^{i_{1},j}\alpha_{1,i_{2}p}=(-1)^{i_{1}+j}\left(
                                                        \begin{array}{c}
                                                          (i_{1}+i_{2})p \\
                                                          i_{1}p \\
                                                        \end{array}
                                                      \right)
\alpha_{0,1}^{i_{1}+i_{2},j},\quad
\alpha_{0}^{i_{1},j}\alpha_{i_{2}p}=\left(
                                                        \begin{array}{c}
                                                          (i_{1}+i_{2})p \\
                                                          i_{1}p \\
                                                        \end{array}
                                                      \right)
\alpha_{0}^{i_{1}+i_{2},j},$$
$$\alpha_{1,k_{1}}\alpha_{k_{2}}=\left(
                                                        \begin{array}{c}
                                                          k_{1}+k_{2} \\
                                                          k_{1} \\
                                                        \end{array}
                                                      \right)
\alpha_{1,k_{1}+k_{2}},\quad
\alpha_{1,ip}\alpha^{s,t}=\left(
                                                        \begin{array}{c}
                                                          (i+s)p-1 \\
                                                          ip \\
                                                        \end{array}
                                                      \right)
\alpha_{1}^{i+s,t},$$
$$\alpha_{k_{1}}\alpha_{k_{2}}=\left(
                                                        \begin{array}{c}
                                                          k_{1}+k_{2} \\
                                                          k_{1} \\
                                                        \end{array}
                                                      \right)
\alpha_{k_{1}+k_{2}},\quad
\alpha_{0,1}^{i_{1},j}\alpha_{i_{2}p}=\left(
                                                        \begin{array}{c}
                                                          (i_{1}+i_{2})p \\
                                                          i_{1}p \\
                                                        \end{array}
                                                      \right)
\alpha_{0,1}^{i_{1}+i_{2},j},$$
$$\alpha_{1}^{s,t}\alpha_{ip}=\left(
                                                        \begin{array}{c}
                                                          (i+s)p-1 \\
                                                          ip \\
                                                        \end{array}
                                                      \right)
\alpha_{1}^{i+s,t},\quad
\alpha_{ip}\alpha^{s,t}=\left(
                                                        \begin{array}{c}
                                                          (i+s)p-1 \\
                                                          ip \\
                                                        \end{array}
                                                      \right)
\alpha^{i+s,t},$$
where $0\leq i,i_{1},i_{2}\leq p^{t_{1}-1}-1,$ $0\leq j\leq p^{t_{2}}-1,$ $0\leq k_{1},k_{2}\leq p^{t_{1}}-1,$ $1\leq s\leq p^{t_{1}-1},$ $1\leq t\leq p^{t_{2}}-1.$

(2)  $\mathrm{DPH}^{\bullet}(\mathcal{F}_{1,2}^{2})\cong \mathcal{U}_{1,2}^{2}(\underline{t}),$
where $\mathcal{U}_{1,2}^{2}(\underline{t})$ is a $p^{t_{1}-1}(4p^{t_{2}}+2p-2)$-dimensional $\mathbb{Z}$-graded superalgebra with a $\mathbb{Z}_{2}$-homogeneous basis
$$\{\alpha_{0}^{i,j},\alpha_{0,1}^{s,k},\alpha_{i,j},\alpha_{0}^{l},\alpha_{l},\alpha_{0,1,s,k}\mid 1\leq i\leq p^{t_{1}-1}, 1\leq j\leq p^{t_{2}}-1
,0\leq s\leq p^{t_{1}-1}-1, $$
$$\quad\quad\quad\quad\quad \quad 0\leq k\leq p^{t_{2}}-1,0\leq l\leq p^{t_{1}}-1\},$$
satisfying that
$$|\alpha_{0}^{i,j}|=|\alpha_{i,j}|=\overline{i+j-1}\ (\mathrm{mod}\ 2),\ |\alpha_{0,1}^{s,k}|=|\alpha_{0,1,s,k}|=\overline{s+k}\ (\mathrm{mod}\ 2),\ |\alpha_{0}^{l}|=|\alpha_{l}|=\overline{l}\ (\mathrm{mod}\ 2);$$
$$\| \alpha_{0}^{i,j}\|=ip+j,\ \|\alpha_{i,j}\|=ip+j-1,\ \|\alpha_{0,1}^{s,k}\|=sp+k+1,\ \|\alpha_{0,1,s,k}\|=sp+k+2,\ \|\alpha_{0}^{l}\|=l+1,\ \|\alpha_{l}\|=l,$$
 and the multiplication is given by
$$\alpha_{0}^{i,j}\alpha_{sp}=\left(
                                \begin{array}{c}
                                  (i+s)p-1 \\
                                  sp \\
                                \end{array}
                              \right)
\alpha_{0}^{i+s,j},\quad
\alpha_{0}^{sp}\alpha_{i,j}=\left(
                                \begin{array}{c}
                                  (i+s)p-1 \\
                                  sp \\
                                \end{array}
                              \right)
\alpha_{0}^{i+s,j},$$
$$\alpha_{i,j}\alpha_{sp}=\left(
                                \begin{array}{c}
                                  (i+s)p-1 \\
                                  sp \\
                                \end{array}
                              \right)
\alpha_{i+s,j},\quad
\alpha_{0}^{s_{1}p}\alpha_{0,1}^{s_{2},k}=(-1)^{s_{1}}\left(
                                                        \begin{array}{c}
                                                          (s_{1}+s_{2})p \\
                                                          s_{1}p \\
                                                        \end{array}
                                                      \right)
\alpha_{0,1,s_{1}+s_{2},k},$$
$$\alpha_{0,1}^{s_{1},k}\alpha_{s_{2}p}=\left(
                                                        \begin{array}{c}
                                                          (s_{1}+s_{2})p \\
                                                          s_{1}p \\
                                                        \end{array}
                                                      \right)
\alpha_{0,1}^{s_{1}+s_{2},k},\quad
\alpha_{0,1,s_{1},k}\alpha_{s_{2}p}=\left(
                                                        \begin{array}{c}
                                                          (s_{1}+s_{2})p \\
                                                          s_{1}p \\
                                                        \end{array}
                                                      \right)
\alpha_{0,1,s_{1}+s_{2},k},$$
$$\alpha_{0}^{l_{1}}\alpha_{l_{2}}=\left(
                                     \begin{array}{c}
                                       l_{1}+l_{2} \\
                                       l_{1} \\
                                     \end{array}
                                   \right)
\alpha_{0}^{l_{1}+l_{2}},\quad
\alpha_{l_{1}}\alpha_{l_{2}}=\left(
                                     \begin{array}{c}
                                       l_{1}+l_{2} \\
                                       l_{1} \\
                                     \end{array}
                                   \right)
\alpha_{l_{1}+l_{2}},$$
where $1\leq i\leq p^{t_{1}-1}$, $1\leq j\leq p^{t_{2}}-1$, $0\leq s,s_{1},s_{2}\leq p^{t_{1}-1}-1,$ $0\leq k\leq p^{t_{2}}-1,$ $0\leq l,l_{1},l_{2}\leq p^{t_{1}}-1.$

(3)  $\mathrm{DPH}^{\bullet}(\mathcal{F}_{1,2}^{3})\cong \mathcal{U}_{1,2}^{3}(\underline{t}),$
where $\mathcal{U}_{1,2}^{3}(\underline{t})$ is a $p^{t_{1}-1}(4p^{t_{2}}+1)$-dimensional $\mathbb{Z}$-graded superalgebra with a $\mathbb{Z}_{2}$-homogeneous basis
$$\{\alpha_{0,1}^{i,j},\alpha_{i},\alpha_{i}^{1},\alpha_{0}^{i,k},\alpha_{1}^{s,k},\alpha_{1,s},\alpha^{s,t}\mid 0\leq i\leq p^{t_{1}-1}-1,0\leq j\leq p^{t_{2}}-2, 0\leq k\leq p^{t_{2}}-1,$$
$$\quad\quad\quad \quad\quad 1\leq s\leq p^{t_{1}-1},1\leq t\leq p^{t_{2}}-1\},$$
satisfying that
$$|\alpha_{0,1}^{i,j}|=\overline{i+j}\ (\mathrm{mod}\ 2),\ |\alpha_{i}|=\overline{i}\ (\mathrm{mod}\ 2),\ |\alpha_{i}^{1}|=\overline{i+1}\ (\mathrm{mod}\ 2),$$
$$|\alpha_{0}^{i,k}|=\overline{i+k}\ (\mathrm{mod}\ 2),\ |\alpha_{1}^{s,k}|=\overline{s+k-1}\ (\mathrm{mod} 2),\ |\alpha_{1,s}|=\overline{s}\ (\mathrm{mod}\ 2),\
|\alpha^{s,t}|=\overline{s+t-1}\ (\mathrm{mod}\ 2);$$
$$ \|\alpha_{0,1}^{i,j}\|=ip+j+2,\ \|\alpha_{i}\|=ip,\ \|\alpha_{i}^{1}\|=ip+1,\ \|\alpha_{0}^{i,k}\|=ip+k+1,$$
$$ \|\alpha_{1}^{s,k}\|=sp+k,\ \|\alpha_{1,s}\|=sp-1,\ \|\alpha^{s,t}\|=sp+t-1,$$
 and the multiplication is given by
$$\alpha_{0,1}^{i_{1},j}\alpha_{i_{2}}=\left(
                                         \begin{array}{c}
                                           (i_{1}+i_{2})p \\
                                           i_{1}p \\
                                         \end{array}
                                       \right)
\alpha_{0,1}^{i_{1}+i_{2},j},\quad
\alpha_{i_{1}}\alpha_{i_{2}}=\left(
                                         \begin{array}{c}
                                           (i_{1}+i_{2})p \\
                                           i_{1}p \\
                                         \end{array}
                                       \right)
\alpha_{i_{1}+i_{2}},$$
$$\alpha_{i_{1}}\alpha_{i_{2}}^{1}=\left(
                                         \begin{array}{c}
                                           (i_{1}+i_{2})p \\
                                           i_{1}p \\
                                         \end{array}
                                       \right)
\alpha_{i_{1}+i_{1}}^{1},\quad
\alpha_{0}^{i_{1},k}\alpha_{i_{2}}=\left(
                                         \begin{array}{c}
                                           (i_{1}+i_{2})p \\
                                           i_{1}p \\
                                         \end{array}
                                       \right)
\alpha_{0}^{i_{1}+i_{2},k},$$
$$\alpha_{1}^{s,k}\alpha_{i}=\left(
                                         \begin{array}{c}
                                           (i+s)p-1 \\
                                           ip \\
                                         \end{array}
                                       \right)
\alpha_{1}^{i+s,k},\quad
\alpha_{1,s}\alpha_{i}=\left(
                                         \begin{array}{c}
                                           (i+s)p-2 \\
                                           ip \\
                                         \end{array}
                                       \right)
\alpha_{1,i+s},$$
$$\alpha^{s,t}\alpha_{i}=\left(
                                         \begin{array}{c}
                                           (i+s)p-1 \\
                                           ip \\
                                         \end{array}
                                       \right)
\alpha^{i+s,t},\quad
\alpha_{1,s}\alpha_{i}^{1}=\left(
                                         \begin{array}{c}
                                           (i+s)p-1 \\
                                           ip+1 \\
                                         \end{array}
                                       \right)
\alpha_{1}^{i+s,0},$$
$$\alpha_{0}^{i,k}\alpha_{1,s}=(-1)^{i+k+1}\left(
                                         \begin{array}{c}
                                           (i+s)p-2 \\
                                           ip \\
                                         \end{array}
                                       \right)
\frac{1}{2}\alpha^{i+s,k+1},$$
where $0\leq i, i_{1},i_{2}\leq p^{t_{1}-1}-1,$ $0\leq j\leq p^{t_{2}}-2,$ $0\leq k\leq p^{t_{2}}-1,$ $1\leq s\leq p^{t_{1}-1},$ $1\leq t\leq p^{t_{2}}-1.$

\begin{proof}
(1) For $k\geq 0$, $0\leq i\leq k$, consider the mapping
\begin{align*}
  d_{2}^{k-i,i}: \mathcal{O}_{k-i}((\mathcal{F}_{1,2}^{1}/\F Y_{2})^{\ast};\underline{t})\bigotimes \F Y_{2}^{\ast^{(i)}}&\longrightarrow \mathcal{O}_{k-i+2}((\mathcal{F}_{1,2}^{1}/\F Y_{2})^{\ast};\underline{t})\bigotimes \F Y_{2}^{\ast^{(i-1)}}, \\
 f\otimes Y_{2}^{\ast^{(i)}}& \longmapsto (-1)^{\|f\|}f\wedge d(Y_{2}^{\ast})\otimes Y_{2}^{\ast^{(i-1)}},
\end{align*}
where
\begin{equation*}
d(Y_{2}^{\ast})=X_{0}^{\ast}\wedge Y_{1}^{\ast}.
\end{equation*}
By Lemma \ref{2.7}, we have
\begin{equation*}
\mathrm{E}^{k-i,i}_{\infty}=\left\{
                \begin{array}{ll}
                 \frac{\mathcal{O}_{k}((\mathcal{F}_{1,2}^{1}/\F Y_{2})^{\ast};\underline{t})}{\mathcal{O}_{k-2}((\mathcal{F}_{1,2}^{1}/\F Y_{2})^{\ast};\underline{t})\bigwedge \F d(Y_{2}^{\ast})}, & \hbox{$i=0$;} \\
                   \frac{\mathrm{Ker}\ d^{k-i,i}_{\infty}}{\mathcal{O}_{k-2-i}((\mathcal{F}_{1,2}^{2}/\F Y_{2})^{\ast};\underline{t})\bigwedge \F d(Y_{2}^{\ast})\bigotimes \F Y_{2}^{\ast^{(i)}}}, & \hbox{$1\leq i\leq k-2$;} \\
                \F X_{0}^{\ast}\bigotimes Y_{2}^{\ast^{(k-1)}}, & \hbox{$i=k-1$;} \\               0, & \hbox{$i=k$.}\\
                \end{array}
              \right.
\end{equation*}
where
$$\mathrm{Ker}\ d^{k-i,i}_{\infty}=\F X_{0}^{\ast}\wedge Y_{1}^{\ast^{(k-i-1)}}\otimes Y_{2}^{\ast^{(i)}}\bigoplus\F X_{0}^{\ast}\wedge X_{1}^{\ast}\wedge Y_{1}^{\ast^{(k-i-2)}}\otimes Y_{2}^{\ast^{(i)}}$$
$$\bigoplus\F \delta_{k-i}X_{1}^{\ast}\wedge Y_{1}^{\ast^{(k-i-1)}}\otimes Y_{2}^{\ast^{(i)}}\bigoplus\F \delta_{k+1-i} Y_{1}^{\ast^{(k-i)}}\otimes Y_{2}^{\ast^{(i)}},$$
where $\delta_{a}=1$ when $a=0$(mod $p$) and $\delta_{a}=0$ otherwise.
From $\mathrm{DPH}^{k}(\mathcal{F}_{1,2}^{1})=\bigoplus\limits_{i=0}^{k}\mathrm{E}^{k-i,i}_{\infty}$, we can obtain the conclusion.

The proofs of $(2), (3)$ are similar to $(1)$.
\end{proof}

\end{thm}

\subsection{Divided power cohomology of $\mathcal{F}_{2,2}$}

\begin{thm}
The following $\mathbb{Z}$-graded superalgebra isomorphisms hold:

(1)  $\mathrm{DPH}^{\bullet}(\mathcal{F}_{2,2}^{1})\cong \mathcal{U}_{2,2}^{1}(\underline{t}),$
where $\mathcal{U}_{2,2}^{1}(\underline{t})$ is a $p^{t_{1}-1}(8p^{t_{2}}+4p-7)$-dimensional $\mathbb{Z}$-graded superalgebra with a $\mathbb{Z}_{2}$-homogeneous basis
$$\{\alpha_{i},\alpha_{0}^{s,j},\alpha_{1}^{i},\alpha_{0,2}^{s,j},\alpha_{0,1}^{s,t},\alpha_{0,1,2}^{s,j},\alpha_{u,t},
\alpha_{1}^{u,v},\alpha_{1,2}^{u,t},\beta_{u,h},\alpha_{1,2}^{i},\beta_{i}\mid 0\leq i\leq p^{t_{1}}-1,0\leq j\leq p^{t_{2}}-1 ,$$
$$0\leq s\leq p^{t_{1}-1}-1,1\leq t\leq p^{t_{2}}-1,1\leq u\leq p^{t_{1}-1},2\leq v\leq p^{t_{2}}-1,1\leq h\leq p^{t_{2}}-2 \},$$
satisfying that
$$|\alpha_{i}|=|\alpha_{1}^{i}|=|\alpha_{1,2}^{i}|=\overline{i}\ (\mathrm{mod}\ 2),\ |\beta_{i}|=\overline{i+1}\ (\mathrm{mod}\ 2),\
|\alpha_{0}^{s,j}|=|\alpha_{0,2}^{s,j}|=|\alpha_{0,1,2}^{s,j}|=\overline{s+j}\ (\mathrm{mod}\ 2),$$
$$ |\alpha_{0,1}^{s,t}|=\overline{s+t}\ (\mathrm{mod}\ 2),\ |\alpha_{u,t}|=|\alpha_{1,2}^{u,t}|=\overline{u+t-1}\ (\mathrm{mod}\ 2),$$
$$ |\alpha_{1}^{u,v}|=\overline{u+v-1}\ (\mathrm{mod} 2),\ |\beta_{u,h}|=\overline{u+h-1}\ (\mathrm{mod}\ 2);$$
 $$\|\alpha_{i}\|=i,\ \|\alpha_{1}^{i}\|=i+1,\ \|\alpha_{1,2}^{i}\|=\|\beta_{i}\|=i+2,\ \|\alpha_{0}^{s,j}\|=sp+j+1,$$
$$\|\alpha_{0,2}^{s,j}\|=sp+j+2,\ \|\alpha_{0,1,2}^{s,j}\|=sp+j+3,\ \|\alpha_{0,1}^{s,t}\|=sp+t+2,$$
$$\|\alpha_{u,t}\|=up+t-1,\ \|\alpha_{1,2}^{u,t}\|=up+t+1,\ \|\alpha_{1}^{u,v}\|=up+v,\ \|\beta_{u,h}\|=up+h,$$
and the multiplication is given by
$$\alpha_{i_{1}}\alpha_{i_{2}}=\left(
                                 \begin{array}{c}
                                  i_{1}+i_{2} \\
                                   i_{1} \\
                                 \end{array}
                               \right)
\alpha_{i_{1}+i_{2}},\quad
\alpha_{1}^{i_{1}}\alpha_{i_{2}}=\left(
                                   \begin{array}{c}
                                     i_{1}+i_{2} \\
                                     i_{1} \\
                                   \end{array}
                                 \right)
\alpha_{1}^{i_{1}+i_{2}},$$
$$\alpha_{1,2}^{i_{1}}\alpha_{i_{2}}=\left(
                                       \begin{array}{c}
                                         i_{1}+i_{2} \\
                                         i_{1}\\
                                       \end{array}
                                     \right)
\alpha_{1,2}^{i_{1}+i_{2}},\quad
\beta_{i_{1}}\alpha_{i_{2}}=\left(
                              \begin{array}{c}
                                i_{1}+i_{2} \\
                                i_{1} \\
                              \end{array}
                            \right)
\beta_{i_{1}+i_{2}},$$
$$\alpha_{0}^{s_{1},j}\alpha_{s_{2}p}=\left(
                                        \begin{array}{c}
                                          (s_{1}+s_{2})p \\
                                          s_{1}p \\
                                        \end{array}
                                      \right)
\alpha_{0}^{s_{1}+s_{2},j},\quad
\alpha_{0,2}^{s_{1},j}\alpha_{s_{2}p}=\left(
                                        \begin{array}{c}
                                          (s_{1}+s_{2})p \\
                                          s_{1}p \\
                                        \end{array}
                                      \right)
\alpha_{0,2}^{s_{1}+s_{2},j},$$
$$\alpha_{0,1}^{s_{1},t}\alpha_{s_{2}p}=\left(
                                          \begin{array}{c}
                                            (s_{1}+s_{2})p \\
                                            s_{1}p\\
                                          \end{array}
                                        \right)
\alpha_{0,1}^{s_{1}+s_{2},t},\quad
\alpha_{0,1,2}^{s_{1},j}\alpha_{s_{2}p}=\left(
                                          \begin{array}{c}
                                            (s_{1}+s_{2})p \\
                                            s_{1}p\\
                                          \end{array}
                                        \right)
\alpha_{0,1,2}^{s_{1}+s_{2},j},$$
$$\alpha_{0,2}^{s_{1},j}\alpha_{s_{2}p+1}=-\left(
                                          \begin{array}{c}
                                            (s_{1}+s_{2})p \\
                                            s_{1}p\\
                                          \end{array}
                                        \right)
\alpha_{0,1}^{s_{1}+s_{2},j+1},\quad
\alpha_{0}^{s_{1},j}\alpha_{1}^{s_{2}p}=(-1)^{s_{1}+j}\left(
                                          \begin{array}{c}
                                            (s_{1}+s_{2})p \\
                                            s_{1}p\\
                                          \end{array}
                                        \right)
\alpha_{0,1}^{s_{1}+s_{2},j},$$
$$\alpha_{0}^{s_{1},j}\alpha_{1,2}^{s_{2}p}=\left(
                                          \begin{array}{c}
                                            (s_{1}+s_{2})p \\
                                            s_{1}p\\
                                          \end{array}
                                        \right)
\alpha_{0,1,2}^{s_{1}+s_{2},j},\quad
\alpha_{1}^{s_{1}p}\alpha_{0,2}^{s_{2},j}=-\left(
                                          \begin{array}{c}
                                            (s_{1}+s_{2})p \\
                                            s_{1}p\\
                                          \end{array}
                                        \right)
\alpha_{0,1,2}^{s_{1}+s_{2},j},$$
$$\alpha_{sp}\alpha_{u,t}=\left(
                            \begin{array}{c}
                              (s+u)p-1 \\
                              sp \\
                            \end{array}
                          \right)
\alpha_{s+u,t},\quad
\alpha_{1}^{u,v}\alpha_{sp}=\left(
                            \begin{array}{c}
                              (s+u)p-1 \\
                              sp \\
                            \end{array}
                          \right)
\alpha_{1}^{s+u,v},$$
$$\alpha_{1,2}^{u,t}\alpha_{sp}=\left(
                            \begin{array}{c}
                              (s+u)p-1 \\
                              sp \\
                            \end{array}
                          \right)
\alpha_{1,2}^{s+u,t},\quad
\beta_{u,h}\alpha_{sp}=\left(
                            \begin{array}{c}
                              (s+u)p-1 \\
                              sp \\
                            \end{array}
                          \right)
\beta_{s+u,h},$$
$$\beta_{u,h}\alpha_{sp+1}=-\left(
                            \begin{array}{c}
                              (s+u)p-1 \\
                              sp \\
                            \end{array}
                          \right)
\alpha_{1}^{s+u,h+1},\quad
\alpha_{1}^{sp}\alpha_{u,t}=\left(
                            \begin{array}{c}
                              (s+u)p-1 \\
                              sp \\
                            \end{array}
                          \right)
\alpha_{1}^{s+u,t},$$
$$\alpha_{1}^{sp}\beta_{u,h}=(-1)^{s}\left(
                            \begin{array}{c}
                              (s+u)p-1 \\
                              sp \\
                            \end{array}
                          \right)
\alpha_{1,2}^{s+u,h},\quad
\alpha_{1,2}^{sp}\alpha_{u,t}=\left(
                            \begin{array}{c}
                              (s+u)p-1 \\
                              sp \\
                            \end{array}
                          \right)
\alpha_{1,2}^{s+u,t},$$
$$\beta_{sp}\alpha_{u,t}=\left(
                            \begin{array}{c}
                              (s+u)p-1 \\
                              sp \\
                            \end{array}
                          \right)
(t+1)\alpha_{1}^{s+u,t+1},\quad
\beta_{u,h}\beta_{sp}=(-1)^{u+h}\left(
                            \begin{array}{c}
                              (s+u)p-1 \\
                              sp \\
                            \end{array}
                          \right)
h\alpha_{1,2}^{u+s,h+1},$$
$$\alpha_{0}^{s_{1},j}\beta_{s_{2}p}=(-1)^{s_{1}+j}\left(
                                          \begin{array}{c}
                                            (s_{1}+s_{2})p \\
                                            s_{1}p\\
                                          \end{array}
                                        \right)
(j+2)\alpha_{0,1}^{s_{1}+s_{2},j+1},$$
$$\alpha_{0,2}^{s_{1},j}\beta_{s_{2}p}=(-1)^{s_{1}+j+1}\left(
                                          \begin{array}{c}
                                            (s_{1}+s_{2})p \\
                                            s_{1}p\\
                                          \end{array}
                                        \right)
(j+1)\alpha_{0,1,2}^{s_{1}+s_{2},j+1},$$
$$\alpha_{1}^{i_{1}}\beta_{i_{2}}=(-1)^{i_{1}+1}\left(
                                                  \begin{array}{c}
                                                    i_{1}+i_{2}+1 \\
                                                    i_{1} \\
                                                  \end{array}
                                                \right)
(i_{2}+1)\alpha_{1,2}^{i_{1}+i_{2}+1},\quad 0\leq i_{1}, i_{2}\leq p^{t_{1}}-1,$$
where $0\leq s,s_{1},s_{2}\leq p^{t_{1}-1}-1$, $0\leq j\leq p^{t_{2}}-1$, $1\leq t\leq p^{t_{2}}-1$, $1\leq u\leq p^{t_{1}-1}$, $2\leq v\leq p^{t_{2}}-1$, $1\leq h\leq p^{t_{2}}-2$.

(2) $\mathrm{DPH}^{\bullet}(\mathcal{F}_{2,2}^{3})\cong \mathcal{U}_{2,2}^{3}(\underline{t}),$
where $\mathcal{U}_{2,2}^{3}(\underline{t})$ is a $p^{t_{1}-1}(8p^{t_{2}}+1)$-dimensional $\mathbb{Z}$-graded superalgebra with a $\mathbb{Z}_{2}$-homogeneous basis
$$\{\alpha_{1}^{i},\alpha_{1}^{i,1},\alpha_{0}^{i,j},\alpha_{0,1}^{i,j},\alpha_{1,2}^{s},\alpha_{1,2}^{s,j},\alpha_{1,s,h},\beta_{s},\beta_{i,t},\beta_{1,i,t},\beta_{2,s,t},\gamma_{i},\gamma_{i}^{1},\gamma_{s,h}\mid 0\leq i\leq p^{t_{1}-1}-1,$$
$$0\leq j\leq p^{t_{2}}-1,1\leq s\leq p^{t_{1}-1},0\leq t\leq p^{t_{2}}-2,1\leq h\leq p^{t_{2}}-1\},$$
satisfying that
$$|\alpha_{1}^{i}|=|\gamma_{i}|=\overline{i}\ (\mathrm{mod}\ 2),\ |\alpha_{1}^{i,1}|=|\gamma_{i}^{1}|=\overline{i+1}\ (\mathrm{mod}\ 2),\
|\alpha_{1,2}^{s}|=|\beta_{s}|=\overline{s}\ (\mathrm{mod}\ 2),$$
$$|\alpha_{0}^{i,j}|=|\alpha_{0,1}^{i,j}|=\overline{i+j}\ (\mathrm{mod}\ 2),\ |\beta_{i,t}|=|\beta_{1,i,t}|=\overline{i+t}\ (\mathrm{mod}\ 2),\
|\beta_{2,s,t}|=\overline{s+t-1}\ (\mathrm{mod} 2),$$
$$|\alpha_{1,2}^{s,j}|=\overline{s+j-1}\ (\mathrm{mod}\ 2),\ |\alpha_{1,s,h}|=|\gamma_{s,h}|=\overline{s+h-1}\ (\mathrm{mod}\ 2);$$
 $$\|\gamma_{i}\|=ip,\ \|\alpha_{1}^{i}\|=\|\gamma_{i}^{1}\|=ip+1,\ \|\alpha_{1}^{i,1}\|=ip+2,\ \|\beta_{s}\|=sp-1,$$
 $$\|\alpha_{1,2}^{s}\|=sp,\ \|\alpha_{0}^{i,j}\|=ip+j+1,\ \|\alpha_{0,1}^{i,j}\|=ip+j+2,\ \|\beta_{i,t}\|=ip+t+2, $$
$$\|\beta_{1,i,t}\|=ip+t+3,\ \|\beta_{2,s,t}\|=sp+t,\ \|\alpha_{1,2}^{s,j}\|=sp+j+1,\ \|\alpha_{1,s,h}\|=sp+h,\ \|\gamma_{s,h}\|=sp+h-1,$$
 and the multiplication is given by
$$\alpha_{1,s,h}\gamma_{i}=\left(
                             \begin{array}{c}
                               (i+s)p-1 \\
                               ip \\
                             \end{array}
                           \right)
\alpha_{1,i+s,h},\quad
\alpha_{1,2}^{s,j}\gamma_{i}=\left(
                             \begin{array}{c}
                               (i+s)p-1 \\
                               ip \\
                             \end{array}
                           \right)
\alpha_{1,2}^{i+s,j},$$
$$\gamma_{s,h}\gamma_{i}=\left(
                             \begin{array}{c}
                               (i+s)p-1 \\
                               ip \\
                             \end{array}
                           \right)\gamma_{i+s,h},
\quad \alpha_{1,2}^{s}\gamma_{i}=\left(
                             \begin{array}{c}
                               (i+s)p-1 \\
                               ip \\
                             \end{array}
                           \right)\alpha_{1,2}^{i+s},$$
$$\beta_{s}\gamma_{i}=\left(
                             \begin{array}{c}
                               (i+s)p-1 \\
                               ip \\
                             \end{array}
                           \right)\beta_{i+s},
\quad \beta_{2,s,t}\gamma_{i}^{1}=\left(
                             \begin{array}{c}
                               (i+s)p-1 \\
                               ip+1\\
                             \end{array}
                           \right)\alpha_{1,i+s,t+1},$$
$$\alpha_{1,2}^{s}\gamma_{i}^{1}=\left(
                             \begin{array}{c}
                               (i+s)p-1 \\
                               ip+1\\
                             \end{array}
                           \right)\alpha_{1,2}^{i+s,0},\quad \beta_{s}\gamma_{i}^{1}=\left(
                             \begin{array}{c}
                               (i+s)p-1 \\
                               ip+1\\
                             \end{array}
                           \right)2\beta_{2,i+s,0},$$
$$\alpha_{1}^{i_{1}}\beta_{i_{2},t}=\left(
                                     \begin{array}{c}
                                       (i_{1}+i_{2})p \\
                                       i_{1}p \\
                                     \end{array}
                                   \right)\beta_{1,i_{1}+i_{2},t},\quad
\alpha_{0}^{i,j}\alpha_{1,2}^{s}=\left(
                             \begin{array}{c}
                               (i+s)p-1 \\
                               ip\\
                             \end{array}
                           \right)\frac{1}{2}\alpha_{1,i+s,j+1},$$
$$\alpha_{0,1}^{i,j}\beta_{s}=(-1)^{i+j+1}\left(
                             \begin{array}{c}
                               (i+s)p-1 \\
                               ip\\
                             \end{array}
                           \right)\alpha_{1,i+s,j+1},
\quad\alpha_{1}^{i}\gamma_{s,h}=\left(
                             \begin{array}{c}
                               (i+s)p-1 \\
                               ip\\
                             \end{array}
                           \right)\alpha_{1,i+s,h},$$
$$\alpha_{1}^{i}\beta_{2,s,t}=(-1)^{i}\left(
                             \begin{array}{c}
                               (i+s)p-1 \\
                               ip\\
                             \end{array}
                           \right)\alpha_{1,2}^{i+s,t},
\quad \alpha_{1}^{i}\beta_{s}=(-1)^{i}\left(
                             \begin{array}{c}
                               (i+s)p-1 \\
                               ip\\
                             \end{array}
                           \right)2\alpha_{1,2}^{i+s},$$
$$\alpha_{1}^{i,1}\beta_{s}=(-1)^{i+1}\left(
                             \begin{array}{c}
                               (i+s)p-1 \\
                               ip+1\\
                             \end{array}
                           \right)2\alpha_{1,2}^{i+s,0},
\quad\beta_{2,s,t}\gamma_{i}=\left(
                             \begin{array}{c}
                               (i+s)p-1 \\
                               ip\\
                             \end{array}
                           \right)\beta_{2,i+s,t},$$
$$\alpha_{1,2}^{s}\beta_{i,t}=\left(
                             \begin{array}{c}
                               (i+s)p-1 \\
                               ip+1\\
                             \end{array}
                           \right)\alpha_{1,2}^{i+s,t+1},
\quad\beta_{i,t}\beta_{s}=(-1)^{i+t}\left(
                             \begin{array}{c}
                               (i+s)p-1 \\
                               ip+1\\
                             \end{array}
                           \right)2\beta_{2,i+s,t+1},$$
$$\gamma_{i_{1}}\gamma_{i_{2}}=\left(
                                 \begin{array}{c}
                                   (i_{1}+i_{2})p \\
                                   i_{1}p \\
                                 \end{array}
                               \right)
\gamma_{i_{1}+i_{2}},\quad
\gamma_{i_{1}}\gamma_{i_{2}}^{1}=\left(
                                 \begin{array}{c}
                                   (i_{1}+i_{2})p \\
                                   i_{1}p \\
                                 \end{array}
                               \right)
\gamma_{i_{1}+i_{2}}^{1},$$
$$\alpha_{0}^{i_{1},j}\gamma_{i_{2}}=\left(
                                 \begin{array}{c}
                                   (i_{1}+i_{2})p \\
                                   i_{1}p \\
                                 \end{array}
                               \right)
\alpha_{0}^{i_{1}+i_{2},j},\quad
\alpha_{0,1}^{i_{1},j}\gamma_{i_{2}}=\left(
                                       \begin{array}{c}
                                         (i_{1}+i_{2})p \\
                                         i_{1}p \\
                                       \end{array}
                                     \right)
\alpha_{0,1}^{i_{1}+i_{2},j},
$$
$$\alpha_{1}^{i_{1}}\gamma_{i_{2}}=\left(
                                     \begin{array}{c}
                                       (i_{1}+i_{2})p \\
                                       i_{1}p \\
                                     \end{array}
                                   \right)
\alpha_{1}^{i_{1}+i_{2}},\quad
\alpha_{1}^{i_{1},1}\gamma_{i_{2}}=\left(
                                     \begin{array}{c}
                                       (i_{1}+i_{2})p \\
                                       i_{1}p \\
                                     \end{array}
                                   \right)
\alpha_{1}^{i_{1}+i_{2},1},
$$
$$\beta_{i_{1},t}\gamma_{i_{2}}=\left(
                                     \begin{array}{c}
                                       (i_{1}+i_{2})p \\
                                       i_{1}p \\
                                     \end{array}
                                   \right)
\beta_{i_{1}+i_{2},t},\quad
\beta_{1,i_{1},t}\gamma_{i_{2}}=\left(
                                     \begin{array}{c}
                                       (i_{1}+i_{2})p \\
                                       i_{1}p \\
                                     \end{array}
                                   \right)
\beta_{1,i_{1}+i_{2},t},
$$
$$\gamma_{i_{1}}^{1}\gamma_{i_{2}}^{1}=2\left(
                                     \begin{array}{c}
                                       (i_{1}+i_{2})p \\
                                       i_{1}p \\
                                     \end{array}
                                   \right)\alpha_{0,1}^{i_{1}+i_{2},0},\quad
\alpha_{1}^{i_{1}}\gamma_{i_{2}}^{1}=\left(
                                     \begin{array}{c}
                                       (i_{1}+i_{2})p \\
                                       i_{1}p \\
                                     \end{array}
                                   \right)\alpha_{1}^{i_{1}+i_{2},1},
$$
$$\beta_{i_{1},t}\gamma_{i_{2}}^{1}=2\left(
                                     \begin{array}{c}
                                       (i_{1}+i_{2})p \\
                                       i_{1}p \\
                                     \end{array}
                                   \right)\alpha_{0,1}^{i_{1}+i_{2},t+1},\quad
\alpha_{0}^{i_{1},j}\alpha_{1}^{i_{2}}=(-1)^{i_{1}+j}\left(
                                     \begin{array}{c}
                                       (i_{1}+i_{2})p \\
                                       i_{1}p \\
                                     \end{array}
                                   \right)\alpha_{0,1}^{i_{1}+i_{2},j},
$$
$$
\alpha_{0}^{i,j}\beta_{s}=(-1)^{i+j+1}\left(
                             \begin{array}{c}
                               (i+s)p-1 \\
                               ip\\
                             \end{array}
                           \right)\gamma_{i+s,j+1},$$
$$\beta_{s}\beta_{1,i,t}=(-1)^{s-1}\left(
                             \begin{array}{c}
                               (i+s)p-1 \\
                               ip+1\\
                             \end{array}
                           \right)2\alpha_{1,2}^{i+s,t+1},$$
$$\beta_{i,t_{1}}\beta_{2,s,t_{2}}=(-1)^{i+t_{1}}\left(
                             \begin{array}{c}
                               (i+s)p-1 \\
                               ip+1\\
                             \end{array}
                           \right)
\left(
                                                 \begin{array}{c}
                                                   t_{1}+t_{2}+1 \\
                                                   t_{1}+1 \\
                                                 \end{array}
                                               \right)
\alpha_{1,i+s,t_{1}+t_{2}+2},
$$
$$\beta_{i_{1},t_{1}}\beta_{i_{2},t_{2}}=2\left(
                                     \begin{array}{c}
                                       (i_{1}+i_{2})p \\
                                       i_{1}p \\
                                     \end{array}
                                   \right)
\left(
                                     \begin{array}{c}
                                       t_{1}+t_{2}+2\\
                                       t_{1}+1 \\
                                     \end{array}
                                   \right)
\alpha_{0,1}^{i_{1}+i_{2},t_{1}+t_{2}+2},
$$
where $0\leq i, i_{1},i_{2}\leq p^{t_{1}-1}-1$, $1\leq s\leq p^{t_{1}-1}$, $1\leq h\leq p^{t_{2}}-1$, $0\leq j\leq p^{t_{2}}-1$, $0\leq t, t_{1}, t_{2}\leq p^{t_{2}}-2$.

(3) $\mathrm{DPH}^{\bullet}(\mathcal{F}_{2,2}^{4})\cong \mathcal{U}_{2,2}^{4}(\underline{t})\ltimes \mathcal{V}_{2,2}^{4}(\underline{t}),$
where $\mathcal{U}_{2,2}^{4}(\underline{t})$ is a $p^{t_{1}-1}(2p+1)$-dimensional $\mathbb{Z}$-graded superalgebra with a $\mathbb{Z}_{2}$-homogeneous basis
$$\{\alpha_{i},\alpha_{0,i},\alpha_{1,s}\mid 0\leq i\leq p^{t_{1}}-1,0\leq s\leq p^{t_{1}-1}-1\},$$
satisfying that
$$|\alpha_{i}|=|\alpha_{0,i}|=\overline{i}\ (\mathrm{mod}\ 2),\ |\alpha_{1,s}|=\overline{s}\ (\mathrm{mod}\ 2);$$
 $$\|\alpha_{i}\|=i,\ \|\alpha_{0,i}\|=i+1,\ \|\alpha_{1,s}\|=sp+1, $$
and the multiplication is given by
$$\alpha_{i_{1}}\alpha_{i_{2}}=\left(
                         \begin{array}{c}
                           i_{1}+i_{2} \\
                           i_{1} \\
                         \end{array}
                       \right)
\alpha_{i_{1}+i_{2}},\quad \alpha_{0,i_{1}}\alpha_{i_{2}}=\left(
                         \begin{array}{c}
                           i_{1}+i_{2} \\
                           i_{1} \\
                         \end{array}
                       \right)\alpha_{0,i_{1}+i_{2}},$$
$$\alpha_{1,s_{1}}\alpha_{s_{2}p}=\left(
                         \begin{array}{c}
                           (s_{1}+s_{2})p \\
                           s_{1}p \\
                         \end{array}
                       \right)\alpha_{1,s_{1}+s_{2}},
\quad\alpha_{1,s}\alpha_{qp+r}=-\left(
                         \begin{array}{c}
                           (s+q)p+r \\
                           sp \\
                         \end{array}
                       \right)\alpha_{0,(s+q)p+r},$$
where $0\leq i_{1},i_{2}\leq p^{t_{1}}-1,$ $0\leq s_{1},s_{2}\leq p^{t_{1}-1}-1,$ $0\leq s\leq p^{t_{1}-1}-1,$ $0\leq qp+r\leq p^{t_{1}}-1,$ $1\leq r\leq p-1.$
$\mathcal{V}_{2,2}^{4}(\underline{t})$ is a $2p^{t_{1}-1}(4p^{t_{2}}+p-3)-1$-dimensional $\mathbb{Z}$-graded superalgebra with a $\mathbb{Z}_{2}$-homogeneous basis
$$\{\alpha_{0,2,i},\beta_{0,2,i},\beta_{0,1,s}^{j},\beta_{0,1,2,s}^{j},\alpha_{0,1,s}^{j},
\alpha_{0,1,2,s}^{k},\alpha_{1,2,s},\alpha_{h}^{j},\alpha_{0,h}^{j},\alpha_{0,2,h}^{j},\beta_{0,2,h}^{\rho},\alpha_{1,2,l}^{1}\mid 0 \leq i\leq p^{t_{1}}-1,$$
$$1\leq j\leq p^{t_{2}}-1,0\leq s\leq p^{t_{1}-1}-1,1\leq h\leq p^{t_{1}-1},0\leq k\leq p^{t_{2}}-1,1\leq l\leq p^{t_{1}-1}-1, 1\leq \rho \leq p^{t_{2}}-2\},$$
satisfying that
$$|\alpha_{0,2,i}|=\overline{i}\ (\mathrm{mod}\ 2),\ |\beta_{0,2,i}|=\overline{i+1}\ (\mathrm{mod}\ 2),\ |\alpha_{1,2,s}|=\overline{s}\ (\mathrm{mod}\ 2),$$
$$|\alpha_{1,2,l}^{1}|=\overline{l+1}\ (\mathrm{mod}\ 2),\
|\alpha_{0,1,2,s}^{k}|=\overline{s+k}\ (\mathrm{mod}\ 2),\
|\beta_{0,2,h}^{\rho}|=\overline{h+\rho-1}\ (\mathrm{mod}\ 2),$$
$$|\beta_{0,1,s}^{j}|=|\beta_{0,1,2,s}^{j}|=|\alpha_{0,1,s}^{j}|=\overline{s+j}\ (\mathrm{mod}\ 2),\
|\alpha_{h}^{j}|=|\alpha_{0,h}^{j}|=|\alpha_{0,2,h}^{j}|=\overline{h+j-1}\ (\mathrm{mod} 2);$$
$$\|\alpha_{0,2,i}\|=\|\beta_{0,2,i}\|=i+2,\
\|\alpha_{1,2,s}\|=sp+2,\
\|\alpha_{1,2,l}^{1}\|=lp+3,\
\|\beta_{0,1,s}^{j}\|=sp+j+1,$$
$$\|\beta_{0,1,2,s}^{j}\|=\|\alpha_{0,1,s}^{j}\|=sp+j+2,\
\|\alpha_{0,1,2,s}^{k}\|=sp+k+3,\
\|\alpha_{h}^{j}\|=hp+j-1,\
\|\alpha_{0,h}^{j}\|=hp+j,$$
$$\|\alpha_{0,2,h}^{j}\|=hp+j+1,\
\|\beta_{0,2,h}^{\rho}\|=hp+\rho,$$
and the multiplication
$$\alpha_{0,2,sp}\alpha_{h}^{j}=\left(
                                  \begin{array}{c}
                                    (s+h)p-1 \\
                                    sp \\
                                  \end{array}
                                \right)
\alpha_{0,2,s+h}^{j},\quad
\alpha_{1,2,s}\alpha_{h}^{j}=-\left(
                                \begin{array}{c}
                                  (s+h)p-1  \\
                                  sp \\
                                \end{array}
                              \right)
\alpha_{0,2,s+h}^{j},$$
$$\beta_{0,2,sp}\alpha_{h}^{j}=\left(
                                \begin{array}{c}
                                  (s+h)p-1  \\
                                  sp \\
                                \end{array}
                              \right)(j+1)\alpha_{0,s+h}^{j+1},
\quad \beta_{0,1,s_{1}}^{j}\alpha_{1,2,s_{2}}=\left(
                                          \begin{array}{c}
                                            (s_{1}+s_{2})p \\
                                            s_{1}p \\
                                          \end{array}
                                        \right)
\alpha_{0,1,2,s_{1}+s_{2}}^{j},$$
$$\beta_{0,2,s_{1}p}\alpha_{1,2,s_{2}}=\left(
                                              \begin{array}{c}
                                                (s_{1}+s_{2})p \\
                                                s_{1}p \\
                                              \end{array}
                                            \right)\alpha_{0,1,2,s_{1}+s_{2}}^{1},$$
$$\alpha_{0,2,s_{1}p}\beta_{0,1,s_{2}}^{j}=(-1)^{s_{1}+1}\left(
                                                           \begin{array}{c}
                                                             (s_{1}+s_{2})p \\
                                                             s_{1}p \\
                                                           \end{array}
                                                         \right)
\alpha_{0,1,2,s_{1}+s_{2}}^{j},$$
$$\beta_{0,1,s_{1}}^{j}\beta_{0,2,s_{2}p}=(-1)^{s_{1}+j+1}\left(
                                                           \begin{array}{c}
                                                             (s_{1}+s_{2})p \\
                                                             s_{1}p \\
                                                           \end{array}
                                                         \right)(j+2)\alpha_{0,1,s_{1}+s_{2}}^{j+1},$$
$$ \beta_{0,2,sp}\beta_{0,2,h}^{\rho}=(-1)^{s+1}\left(
                                                           \begin{array}{c}
                                                             (s+h)p-1 \\
                                                             sp+1 \\
                                                           \end{array}
                                                         \right)\rho\alpha_{0,2,s+h}^{\rho+1},$$
$$\beta_{0,2,s_{1}p}\beta_{0,1,2,s_{2}}^{j}=\left(
                                              \begin{array}{c}
                                                (s_{1}+s_{2})p \\
                                                s_{1}p \\
                                              \end{array}
                                            \right)
(j+1)\alpha_{0,1,2,s_{1}+s_{2}}^{j+1},$$
where $0\leq s,s_{1},s_{2}\leq p^{t_{1}-1}-1,$ $1\leq h\leq p^{t_{1}-1},$ $1\leq j\leq p^{t_{2}}-1,$ $1\leq \rho\leq p^{t_{2}}-2.$
The multiplication between $\mathcal{U}_{2,2}^{4}(\underline{t})$ and $\mathcal{V}_{2,2}^{4}(\underline{t})$ is graded-supercommutative,
and the multiplication is given by
$$\alpha_{1,2,0}\alpha_{1}=-\alpha_{0,1,0}^{1}-\alpha_{0,2,1},\quad
\alpha_{1,2,l}\alpha_{sp+1}=\left(
                              \begin{array}{c}
                                (s+l)p \\
                                sp \\
                              \end{array}
                            \right)
\alpha_{1,2,s+l}^{1},$$
$$ \beta_{0,2,0}\alpha_{1,0}=-2\alpha_{0,1,0}^{1}-\alpha_{0,2,1},\quad
\alpha_{0,s_{1}p}\beta_{0,1,s_{2}}^{j}=(-1)^{s_{1}}\left(
                                                             \begin{array}{c}
                                                               (s_{1}+s_{2})p \\
                                                               s_{1}p \\
                                                             \end{array}
                                                           \right)
\alpha_{0,1,s_{1}+s_{2}}^{j},$$
$$\alpha_{1,2,l}^{1}\alpha_{hp-1}=\left(
                                    \begin{array}{c}
                                      (h+l)p \\
                                      lp+1 \\
                                    \end{array}
                                  \right)
\alpha_{1,2,h+l},\quad
\alpha_{0,hp-1}\alpha_{1,2,l}^{1}=\left(
                                    \begin{array}{c}
                                      (h+l)p \\
                                      lp+1 \\
                                    \end{array}
                                  \right)
\alpha_{0,1,2,h+l}^{0},$$
$$\alpha_{h}^{j}\alpha_{sp}=\left(
                              \begin{array}{c}
                                (s+h)p-1 \\
                                sp \\
                              \end{array}
                            \right)
\alpha_{s+h}^{j},\quad \beta_{0,1,s_{1}}^{j}\alpha_{s_{2}p}=\left(
                                                              \begin{array}{c}
                                                                (s_{1}+s_{2})p \\
                                                                s_{1}p \\
                                                              \end{array}
                                                            \right)
\beta_{0,1,s_{1}+s_{2}}^{j},$$
$$\beta_{0,2,h}^{\rho}\alpha_{sp}=\left(
                                    \begin{array}{c}
                                      (s+h)p-1 \\
                                      sp \\
                                    \end{array}
                                  \right)
\beta_{0,2,s+h}^{\rho},
\quad\beta_{0,2,h}^{\rho}\alpha_{sp+1}=\left(
                                                      \begin{array}{c}
                                                                          (s+h)p-1 \\
                                                                             sp+1 \\
                                                                              \end{array}
                                                                           \right)\alpha_{0,s+h}^{\rho+1},$$
$$ \alpha_{0,h}^{j}\alpha_{sp}=
\left(
                                                      \begin{array}{c}
                                                                          (s+h)p-1 \\
                                                                             sp \\
                                                                              \end{array}
                                                                           \right)
\alpha_{0,s+h}^{j},
\quad \beta_{0,1,2,s_{1}}^{j}\alpha_{s_{2}p}=\left(
                                            \begin{array}{c}
                                              (s_{1}+s_{2})p \\
                                              s_{1}p \\
                                            \end{array}
                                          \right)
\beta_{0,1,2,s_{1}+s_{2}}^{j},$$
$$\beta_{0,1,2,s_{1}}^{j}\alpha_{s_{2}p+1}=-\left(
                                              \begin{array}{c}
                                                (s_{1}+s_{2})p \\
                                                s_{1}p \\
                                              \end{array}
                                            \right)
\alpha_{0,1,s_{1}+s_{2}}^{j+1},
\quad \alpha_{0,2,h}^{j}\alpha_{sp}=\left(
                                      \begin{array}{c}
                                        (s+h)p-1 \\
                                        sp \\
                                      \end{array}
                                    \right)
\alpha_{0,2,s+h}^{j},$$
$$\alpha_{0,1,s_{1}}^{j}\alpha_{s_{2}p}=\left(
                                          \begin{array}{c}
                                            (s_{1}+s_{2})p \\
                                            s_{1}p \\
                                          \end{array}
                                        \right)
\alpha_{0,1,s_{1}+s_{2}}^{j},
\quad\alpha_{0,1,2,s_{1}}^{k}\alpha_{s_{2}p}=\left(
                                \begin{array}{c}
                                 (s_{1}+s_{2})p \\
                                            s_{1}p \\
                                 \end{array}
                                  \right)\alpha_{0,1,2,s_{1}+s_{2}}^{k},$$
$$ \alpha_{0,2,i_{1}}\alpha_{i_{2}}=\left(
                                \begin{array}{c}
                                 i_{1}+i_{2} \\
                                 i_{1}\\
                                 \end{array}
                                  \right)
\alpha_{0,2,i_{1}+i_{2}},\quad \alpha_{1,2,s_{1}}\alpha_{s_{2}p}=\left(
                                \begin{array}{c}
                                 (s_{1}+s_{2})p \\
                                            s_{1}p \\
                                 \end{array}
                                  \right)\alpha_{1,2,s_{1}+s_{2}},$$
$$\alpha_{1,2,s}\alpha_{lp+1}=\left(
                                \begin{array}{c}
                                  (s+l)p \\
                                  sp \\
                                \end{array}
                              \right)
\alpha_{1,2,s+l}^{1},\quad \alpha_{1,2,s}\alpha_{q_{1}p+r_{1}}=-\left(
                                                                  \begin{array}{c}
                                                                    (s+q_{1})p+r_{1} \\
                                                                    sp \\
                                                                  \end{array}
                                                                \right)
\alpha_{0,2,(s+q_{1})p+r_{1}},$$
$$\alpha_{1,2,l}^{1}\alpha_{sp}=\left(
                                  \begin{array}{c}
                                    (s+l)p \\
                                    sp \\
                                  \end{array}
                                \right)
\alpha_{1,2,s+l}^{1},
\quad \alpha_{1,2,l}^{1}\alpha_{q_{2}p+r_{2}}=-\left(
                                                 \begin{array}{c}
                                                   (q_{2}+l)p+r+1 \\
                                                   lp+1 \\
                                                 \end{array}
                                               \right)
\alpha_{0,2,(q_{2}+l)p+r_{2}+1},$$
$$\quad \beta_{0,2,i_{1}}\alpha_{i_{2}}=\left(
                                  \begin{array}{c}
                                    i_{1}+i_{2} \\
                                    i_{1} \\
                                  \end{array}
                                \right)
\beta_{0,2,i_{1}+i_{2}},\quad
\alpha_{0,sp}\alpha_{h}^{j}=\left(
                              \begin{array}{c}
                                (s+h)p-1 \\
                                sp \\
                              \end{array}
                            \right)
\alpha_{0,s+h}^{j},$$
$$\alpha_{0,s_{1}p}\beta_{0,1,2,s_{2}}^{j}=\left(
                                             \begin{array}{c}
                                               (s_{1}+s_{2})p \\
                                               s_{1}p \\
                                             \end{array}
                                           \right)
\alpha_{0,1,2,s_{1}+s_{2}}^{j},\quad \alpha_{0,s_{1}p}\alpha_{1,2,s_{2}}=\left(
                                             \begin{array}{c}
                                               (s_{1}+s_{2})p \\
                                               s_{1}p \\
                                             \end{array}
                                           \right)\alpha_{0,1,2,s_{1}+s_{2}}^{0},$$
$$\quad \alpha_{1,s}\alpha_{h}^{j}=-\left(
                                      \begin{array}{c}
                                        (s+h)p-1 \\
                                        sp \\
                                      \end{array}
                                    \right)
\alpha_{0,s+h}^{j},
\quad\alpha_{1,s_{1}}\alpha_{0,2,s_{2}p}=-\left(
                                                                         \begin{array}{c}
                                                                           (s_{1}+s_{2})p \\
                                                                           s_{1}p \\
                                                                         \end{array}
                                                                       \right)
\alpha_{0,1,2,s_{1}+s_{2}}^{0},$$
$$\beta_{0,2,lp}\alpha_{1,s}=(-1)^{l+1}\left(
                                         \begin{array}{c}
                                           (s+l)p \\
                                           sp \\
                                         \end{array}
                                       \right)
\alpha_{0,1,s+l}^{1}+(-1)^{l}\left(
                                         \begin{array}{c}
                                           (s+l)p+1 \\
                                           sp \\
                                         \end{array}
                                       \right)\alpha_{1,2,s+l}^{1},$$
$$\beta_{0,2,sp}\alpha_{1,l}=(-1)^{s+1}\left(
                                         \begin{array}{c}
                                           (s+l)p \\
                                           sp \\
                                         \end{array}
                                       \right)
\alpha_{0,1,s+l}^{1}+(-1)^{s}\left(
                                         \begin{array}{c}
                                           (s+l)p+1 \\
                                           lp \\
                                         \end{array}
                                       \right)\alpha_{1,2,s+l}^{1},$$
$$\alpha_{0,sp}\beta_{0,2,h}^{\rho}=(-1)^{s+1}\left(
                                                                                 \begin{array}{c}
                                                                                   (s+h)p-1 \\
                                                                                   sp \\
                                                                                 \end{array}
                                                                               \right)
\alpha_{0,2,s+h}^{\rho},$$
$$\alpha_{0,i_{1}}\beta_{0,2,i_{2}}=(-1)^{i_{1}}(i_{2}+1)\left(
                                                           \begin{array}{c}
                                                             i_{1}+i_{2}+1 \\
                                                             i_{1} \\
                                                           \end{array}
                                                         \right)
\alpha_{0,2,i_{1}+i_{2}+1},$$
$$\beta_{0,2,h}^{\rho}\alpha_{1,s}=(-1)^{h+\rho}\left(
                                                  \begin{array}{c}
                                                    (s+h)p-1 \\
                                                    sp \\
                                                  \end{array}
                                                \right)
\alpha_{0,2,s+h}^{\rho},$$
$$ \beta_{0,1,s_{1}}^{j}\alpha_{1,s_{2}}=(-1)^{s_{1}+j}\left(
                                                                                       \begin{array}{c}
                                                                                         (s_{1}+s_{2})p \\
                                                                                         s_{1}p \\
                                                                                       \end{array}
                                                                                     \right)
\alpha_{0,1,s_{1}+s_{2}}^{j},$$
$$\beta_{0,1,2,s_{1}}^{j}\alpha_{1,s_{2}}=(-1)^{s_{1}+j+1}\left(
                                             \begin{array}{c}
                                               (s_{1}+s_{2})p \\
                                               s_{1}p \\
                                             \end{array}
                                           \right)
\alpha_{0,1,2,s_{1}+s_{2}}^{j},$$
$$\beta_{0,2,q_{3}p+r_{3}}\alpha_{1,s}=(-1)^{q_{3}+r_{3}+1}(r_{3}+1)\left(
                                                                      \begin{array}{c}
                                                                        (s+q_{3})p+r_{3}+1 \\
                                                                        q_{3}p+r_{3}+1 \\
                                                                      \end{array}
                                                                    \right)
\alpha_{0,2,(s+q_{3})p+r_{3}+1},$$
where $0\leq i_{1},i_{2},q_{1}p+r_{1},q_{2}p+r_{2},q_{3}p+r_{3}\leq p^{t_{1}}-1,$ $0\leq s,s_{1},s_{2}\leq p^{t_{1}-1}-1,$ $1\leq h\leq p^{t_{1}-1},$ $0\leq k\leq p^{t_{2}}-1,$
$1\leq l\leq p^{t_{1}-1}-1,$ $1\leq j\leq p^{t_{2}}-1,$ $1\leq \rho\leq p^{t_{2}}-2,$ $2\leq r_{1}\leq p-1,$ $1\leq r_{2},r_{3}\leq p-2.$

(4) $\mathrm{DPH}^{\bullet}(\mathcal{F}_{2,2}^{5})\cong \mathcal{U}_{2,2}^{5}(\underline{t}),$
where $\mathcal{U}_{2,2}^{5}(\underline{t})$ is a $p^{t_{1}-1}(8p^{t_{2}}+2)$-dimensional $\mathbb{Z}$-graded superalgebra with a $\mathbb{Z}_{2}$-homogeneous basis
$$\{\alpha_{i},\alpha_{1}^{i},\alpha_{0,1}^{i},\alpha_{0}^{i},\alpha_{0}^{k,l},\alpha_{0,1}^{i,j},\alpha_{0,2}^{k,j},\alpha_{0,2}^{k},\alpha_{k,l},\beta_{k},
\beta_{0,1}^{i,j},\beta_{1,i,s},\beta_{2,i,s},\beta_{3,k,j}\mid 0\leq i\leq p^{t_{1}-1}-1,$$
$$\quad\quad\quad \quad\quad\quad \quad\quad \quad\quad\quad0\leq j\leq p^{t_{2}}-1,0\leq s\leq p^{t_{2}}-2,1\leq k\leq p^{t_{1}-1},1\leq l\leq p^{t_{2}}-1\},$$
satisfying that
$$|\alpha_{i}|=|\alpha_{0}^{i}|=\overline{i}\ (\mathrm{mod}\ 2),\ |\alpha_{0,1}^{i}|=|\alpha_{1}^{i}|=\overline{i+1}\ (\mathrm{mod}\ 2),\
 |\alpha_{0,1}^{i,j}|=|\beta_{0,1}^{i,j}|=\overline{i+j}\ (\mathrm{mod}\ 2),$$
$$|\beta_{1,i,s}|=|\beta_{2,i,s}|=\overline{i+s}\ (\mathrm{mod}\ 2),\ |\alpha_{0}^{k,l}|=|\alpha_{k,l}|=\overline{k+l-1}\ (\mathrm{mod}\ 2),$$
$$|\alpha_{0,2}^{k,j}|=|\beta_{3,k,j}|=\overline{k+j-1}\ (\mathrm{mod}\ 2),\ |\alpha_{0,2}^{k}|=|\beta_{k}|=\overline{k}\ (\mathrm{mod}\ 2);$$
$$\|\alpha_{i}\|=ip,\ \|\alpha_{0}^{i}\|=\|\alpha_{1}^{i}\|=ip+1,\ \|\alpha_{0,1}^{i}\|=ip+2,\ \|\alpha_{0,1}^{i,j}\|=ip+j+2,$$
 $$\|\beta_{0,1}^{i,j}\|=ip+j+1,\ \|\beta_{1,i,s}\|=ip+s+2,\ \|\beta_{2,i,s}\|=ip+s+3,\ \|\alpha_{0}^{k,l}\|=kp+l,$$
$$\|\alpha_{k,l}\|=kp+l-1,\ \|\alpha_{0,2}^{k,j}\|=kp+j+1,\ \|\beta_{3,k,j}\|=kp+j,\ \|\alpha_{0,2}^{k}\|=kp,\ \|\beta_{k}\|=kp-1,$$
 and the multiplication is given by
$$\alpha_{0,1}^{i_{1},j}\alpha_{i_{2}}=\left(
                                         \begin{array}{c}
                                           (i_{1}+i_{2})p \\
                                           i_{1}p \\
                                         \end{array}
                                       \right)
\alpha_{0,1}^{i_{1}+i_{2},j},\quad
\alpha_{0}^{i_{1}}\beta_{0,1}^{i_{2},j}=(-1)^{i_{1}}\left(
                                         \begin{array}{c}
                                           (i_{1}+i_{2})p \\
                                           i_{1}p \\
                                         \end{array}
                                       \right)
\alpha_{0,1}^{i_{1}+i_{2},j},$$
$$\beta_{0,1}^{i_{1},j}\alpha_{i_{2}}=\left(
                                         \begin{array}{c}
                                           (i_{1}+i_{2})p \\
                                           i_{1}p \\
                                         \end{array}
                                          \right)
\beta_{0,1}^{i_{1}+i_{2},j},\quad
\alpha_{0}^{i_{1}}\alpha_{i_{2}}=\left(
                                         \begin{array}{c}
                                           (i_{1}+i_{2})p \\
                                           i_{1}p \\
                                         \end{array}
                                       \right)
\alpha_{0}^{i_{1}+i_{2}},$$
$$\alpha_{0}^{i_{1}}\alpha_{1}^{i_{2}}=\left(
                                         \begin{array}{c}
                                           (i_{1}+i_{2})p \\
                                           i_{1}p \\
                                         \end{array}
                                       \right)
\alpha_{0,1}^{i_{1}+i_{2}},\quad
\alpha_{0}^{i_{1}}\beta_{1,i_{2},s}=\left(
                                         \begin{array}{c}
                                           (i_{1}+i_{2})p \\
                                           i_{1}p \\
                                         \end{array}
                                       \right)
\beta_{2,i_{1}+i_{2},s},$$
$$\alpha_{0,1}^{i_{1}}\alpha_{i_{2}}=\left(
                                         \begin{array}{c}
                                           (i_{1}+i_{2})p \\
                                           i_{1}p \\
                                         \end{array}
                                       \right)
\alpha_{0,1}^{i_{1}+i_{2}},\quad
\alpha_{i_{1}}\alpha_{i_{2}}=\left(
                                         \begin{array}{c}
                                           (i_{1}+i_{2})p \\
                                           i_{1}p \\
                                         \end{array}
                                       \right)
\alpha_{i_{1}+i_{2}},$$
$$\alpha_{i_{1}}\alpha_{1}^{i_{2}}=\left(
                                         \begin{array}{c}
                                           (i_{1}+i_{2})p \\
                                           i_{1}p \\
                                         \end{array}
                                       \right)
\alpha_{1}^{i_{1}+i_{2}},\quad
\beta_{1,i_{1},s}\alpha_{i_{2}}=\left(
                                         \begin{array}{c}
                                           (i_{1}+i_{2})p \\
                                           i_{1}p \\
                                         \end{array}
                                       \right)
\beta_{1,i_{1}+i_{2},s},$$
$$\beta_{2,i_{1},s}\alpha_{i_{2}}=\left(
                                         \begin{array}{c}
                                           (i_{1}+i_{2})p \\
                                           i_{1}p \\
                                         \end{array}
                                       \right)
\beta_{2,i_{1}+i_{2},s},\quad
\alpha_{0,1}^{i,j}\beta_{k}=(-1)^{i+j+1}\left(
                                          \begin{array}{c}
                                            (i+k)p-2 \\
                                            ip \\
                                          \end{array}
                                        \right)
\frac{1}{2}\alpha_{0}^{i+k,j+1},$$
$$\alpha_{0,2}^{k}\alpha_{i}=\left(
                                          \begin{array}{c}
                                            (i+k)p-2 \\
                                            ip \\
                                          \end{array}
                                        \right)
\alpha_{0,2}^{i+k},\quad
\alpha_{0,2}^{k}\beta_{0,1}^{i,j}=(-1)^{k}\left(
                                          \begin{array}{c}
                                            (i+k)p-2 \\
                                            ip \\
                                          \end{array}
                                        \right)
\frac{1}{2}\alpha_{0}^{i+k,j+1},$$
$$\beta_{k}\alpha_{i}=\left(
                                          \begin{array}{c}
                                            (i+k)p-2 \\
                                            ip \\
                                          \end{array}
                                        \right)
\beta_{i+k},\quad
\beta_{0,1}^{i,j}\beta_{k}=(-1)^{i+j+1}\left(
                                          \begin{array}{c}
                                            (i+k)p-2 \\
                                            ip \\
                                          \end{array}
                                        \right)
\frac{1}{2}\alpha_{i+k,j+1},$$
$$\alpha_{0}^{i}\alpha_{k,l}=\left(
                                          \begin{array}{c}
                                            (i+k)p-1 \\
                                            ip \\
                                          \end{array}
                                        \right)
\alpha_{0}^{i+k,l},\quad
\alpha_{0}^{i}\beta_{k}=(-1)^{i}\left(
                                          \begin{array}{c}
                                            (i+k)p-2 \\
                                            ip \\
                                          \end{array}
                                        \right)
\alpha_{0,2}^{i+k},$$
$$\alpha_{0}^{k,l}\alpha_{i}=\left(
                                          \begin{array}{c}
                                            (i+k)p-1 \\
                                            ip \\
                                          \end{array}
                                        \right)
\alpha_{0}^{i+k,l},\quad
\alpha_{0,2}^{k,j}\alpha_{i}=\left(
                                          \begin{array}{c}
                                            (i+k)p-1 \\
                                            ip \\
                                          \end{array}
                                        \right)
\alpha_{0,2}^{i+k,j},$$
$$\alpha_{i}\alpha_{k,l}=\left(
                                          \begin{array}{c}
                                            (i+k)p-1 \\
                                            ip \\
                                          \end{array}
                                        \right)
\alpha_{i+k,l},\quad
\beta_{3,k,j}\alpha_{i}=\left(
                                          \begin{array}{c}
                                            (i+k)p-1 \\
                                            ip \\
                                          \end{array}
                                        \right)
\beta_{3,i+k,j},$$
$$\alpha_{0}^{i}\beta_{3,k,j}=(-1)^{i}\left(
                                          \begin{array}{c}
                                            (i+k)p-1 \\
                                            ip \\
                                          \end{array}
                                        \right)
\alpha_{0,2}^{i+k,j},\quad
\beta_{k}\alpha_{1}^{i}=\left(
                                          \begin{array}{c}
                                            (i+k)p-1 \\
                                            ip+1 \\
                                          \end{array}
                                        \right)
\beta_{3,i+k,0},$$
$$\alpha_{0,1}^{i}\beta_{k}=(-1)^{i+1}\left(
                                          \begin{array}{c}
                                            (i+k)p-1 \\
                                            ip+1 \\
                                          \end{array}
                                        \right)
\alpha_{0,2}^{i+k,0},\quad
\alpha_{0,2}^{k}\alpha_{1}^{i}=\left(
                                          \begin{array}{c}
                                            (i+k)p-1 \\
                                            ip+1 \\
                                          \end{array}
                                        \right)
\alpha_{0,2}^{i+k,0},$$
$$\alpha_{0,2}^{k}\beta_{1,i,s}=\left(
                                          \begin{array}{c}
                                            (i+k)p-1 \\
                                            ip+1 \\
                                          \end{array}
                                        \right)
\alpha_{0,2}^{i+k,s+1},\quad
\beta_{k}\beta_{1,i,s}=\left(
                                          \begin{array}{c}
                                            (i+k)p-1 \\
                                            ip+1 \\
                                          \end{array}
                                        \right)
\beta_{3,i+k,s+1},$$
$$\beta_{2,i,s}\beta_{k}=(-1)^{i+s}\left(
                                          \begin{array}{c}
                                            (i+k)p-1 \\
                                            ip+1 \\
                                          \end{array}
                                        \right)
\alpha_{0,2}^{i+k,s+1},\quad
\alpha_{1}^{i_{1}}\alpha_{1}^{i_{2}}=\left(
                                          \begin{array}{c}
                                            (i_{1}+i_{2})p+2 \\
                                            i_{1}p+1 \\
                                          \end{array}
                                        \right)
\alpha_{0,1}^{i_{1}+i_{2},0},$$
$$\beta_{1,i_{1},s}\alpha_{1}^{i_{2}}=\left(
                                          \begin{array}{c}
                                            (i_{1}+i_{2})p+2 \\
                                            i_{1}p+1 \\
                                          \end{array}
                                        \right)
\alpha_{0,1}^{i_{1}+i_{2},s+1},\quad
\beta_{3,k,j}\alpha_{1}^{i}=-\left(
                                          \begin{array}{c}
                                            (i+k)p-2 \\
                                            ip+1 \\
                                          \end{array}
                                        \right)
\frac{1}{2}\alpha_{0}^{i+k,j+1},$$
$$\beta_{1,i_{1},s_{1}}\beta_{1,i_{2},s_{2}}=\left(
                                          \begin{array}{c}
                                            (i_{1}+i_{2})p+2\\
                                            i_{1}p+1 \\
                                          \end{array}
                                        \right)
\left(
                                          \begin{array}{c}
                                            s_{1}+s_{2}+2 \\
                                            s_{1}+1 \\
                                          \end{array}
                                        \right)
\alpha_{0,1}^{i_{1}+i_{2},s_{1}+s_{2}+2},$$
$$\beta_{3,k,j}\beta_{1,i,s}=-\left(
                                          \begin{array}{c}
                                            (i+k)p-2 \\
                                            ip+1 \\
                                          \end{array}
                                        \right)
\left(
                                          \begin{array}{c}
                                            j+s+1 \\
                                            j \\
                                          \end{array}
                                        \right)
\frac{1}{2}\alpha_{0}^{i+k,j+s+2},$$
where $0 \leq i,i_{1},i_{2}\leq p^{t_{1}-1}-1,$ $0\leq j\leq p^{t_{2}}-1,$ $0\leq s,s_{1},s_{2}\leq p^{t_{2}}-2,$
$1\leq k\leq p^{t_{1}-1},$ $1\leq l\leq p^{t_{2}}-1.$

\end{thm}
\begin{proof}
The proof is similar to Theorem \ref{3}.
\end{proof}

\small\noindent \textbf{Acknowledgment}\\
The first author was supported by the NSF of China (11501151) and the NSF of Heilongjiang Province (A2015003, A2017005). The corresponding author was supported by the NSF of China (11471090, 11701158) and the NSF of Heilongjiang Province (A2015017). The authors appreciate the referee for his/her valuable suggestions.

\end{document}